\documentclass[10pt]{amsart}
\usepackage{latexsym}
\usepackage{amsmath}
\usepackage{amssymb}
\usepackage{mathrsfs}
\usepackage{graphicx}
\usepackage{color}
\usepackage{pdfpages}
\usepackage{ifthen}
\usepackage{leftidx,tensor}
\usepackage[T1]{fontenc}
\usepackage[latin1]{inputenc}
\usepackage{mathtools}
\usepackage{comment}
\usepackage{dsfont}
\usepackage[nocompress]{cite}

\usepackage[shortlabels]{enumitem}
\usepackage{aliascnt}
\usepackage[bookmarks=true,pdfstartview=FitH, pdfborder={0 0 0}, colorlinks=true,citecolor=red, linkcolor=blue]{hyperref}
\usepackage{bbm}
\usepackage{nicefrac}

\theoremstyle{plain}
\newtheorem{thm}{Theorem}[section]
\newaliascnt{cor}{thm}
\newaliascnt{prop}{thm}
\newaliascnt{lem}{thm}
\newaliascnt{rmk}{thm}
\newtheorem{cor}[cor]{Corollary}
\newtheorem{prop}[prop]{Proposition}
\newtheorem{lem}[lem]{Lemma}
\newtheorem{rmk}[rmk]{Remark}
\aliascntresetthe{cor}
\aliascntresetthe{prop}
\aliascntresetthe{lem} 
\aliascntresetthe{rmk} 
\theoremstyle{definition}
\newaliascnt{defn}{thm}
\newaliascnt{asu}{thm}
\newaliascnt{con}{thm}

\newtheorem{asu}[asu]{Assumption}

\aliascntresetthe{defn}
\aliascntresetthe{asu}
\aliascntresetthe{con}
\newcounter{stp}
\newcounter{stpi}
\newcounter{stpci}
\newcounter{stpiii}

\theoremstyle{remark}
\newaliascnt{rem}{thm}
\newaliascnt{exa}{thm}
\newaliascnt{masu}{thm}
\newaliascnt{nota}{thm}
\newaliascnt{sett}{thm}
\aliascntresetthe{rem}
\aliascntresetthe{exa}
\aliascntresetthe{masu}
\aliascntresetthe{nota}
\aliascntresetthe{sett}
\numberwithin{equation}{section}

\setlist[enumerate]{font = \normalfont}

\renewcommand {\S}	{\mathbb{S}}
\newcommand {\N}	{\mathbb{N}}

\newcommand {\R}	{\mathbb{R}}

\newcommand {\E}	{\mathbb{E}}
\newcommand {\F}	{\mathbb{F}}

\newcommand{\diag}{\mathrm{diag}}
\renewcommand{\d}{\, \mathrm{d}}
\DeclareMathOperator{\re}{Re}
\DeclareMathOperator{\im}{Im}
\DeclareMathOperator{\Id}{Id}

\newcommand{\hra}{\hookrightarrow}

\newcommand{\Hinfty}{\mathcal{H}^\infty}

\renewcommand{\div}{\mathrm{div}}

\newcommand{\sA}{\mathcal{A}}

\newcommand{\sF}{\mathcal{F}}

\newcommand{\sL}{\mathcal{L}}
\newcommand{\sT}{\mathcal{T}}
\newcommand{\sS}{\mathcal{S}}
\newcommand{\sJ}{\mathcal{J}}

\newcommand{\tB}{\tilde{\mathrm{B}}}

\newcommand{\tD}{\tilde{\mathrm{D}}}

\newcommand{\tF}{\tilde{F}}

\newcommand{\tL}{\tilde{\mathrm{L}}}

\newcommand{\tN}{\tilde{N}}

\newcommand{\tX}{\tilde{X}}
\newcommand{\tu}{\tilde{u}}
\newcommand{\tv}{\tilde{v}}
\newcommand{\tw}{\tilde{w}}

\newcommand{\tm}{\tilde{m}}
\newcommand{\tn}{\tilde{n}}
\newcommand{\thseaice}{\tilde{h}}
\newcommand{\ta}{\tilde{a}}

\newcommand{\tb}{\tilde{b}}

\newcommand{\tepsilon}{\tilde{\eps}}
\newcommand{\tsA}{\tilde{\mathcal{A}}}

\newcommand{\tsG}{\tilde{\mathcal{G}}}

\newcommand{\tsM}{\tilde{\mathcal{M}}}
\newcommand{\tsF}{\tilde{\mathcal{F}}}
\newcommand{\tsX}{\tilde{\mathcal{X}}}
\newcommand{\tsY}{\tilde{\mathcal{Y}}}

\newcommand{\eps}{\varepsilon}

\newcommand{\tri}{\triangle}
\newcommand{\trid}{\triangle_\delta}
\newcommand{\Itwo}{\mathrm{I}_2}
\newcommand{\sigd}{\sigma_\delta}

\newcommand{\T}{\mathrm{T}}

\newcommand{\dist}{\mathrm{dist}}

\renewcommand{\phi}{\varphi}

\newcommand{\cB}{\mathcal{B}}
\newcommand{\cD}{\mathcal{D}}
\newcommand{\cQ}{\mathcal{Q}}
\newcommand{\cO}{\mathcal{O}}
\newcommand{\cQD}{\cQ_\cD}
\newcommand{\cQG}{\cQ_\Gamma}
\newcommand{\cQpO}{\cQ_{\partial \mathcal{O}}}

\newcommand{\ccor}{c_{\mbox{\tiny{cor}}}}
\newcommand{\Catm}{C_{\mbox{\tiny{atm}}}}
\newcommand{\Cocean}{C_{\mbox{\tiny{ocean}}}}
\newcommand{\Uatm}{U_{\mbox{\tiny{atm}}}}
\newcommand{\Uocean}{U_{\mbox{\tiny{ocean}}}}
\newcommand{\Ratm}{R_{\mbox{\tiny{atm}}}}
\newcommand{\Rocean}{R_{\mbox{\tiny{ocean}}}}
\newcommand{\ratm}{\rho_{\mbox{\tiny{atm}}}}
\newcommand{\rice}{\rho_{\mbox{\tiny{ice}}}}
\newcommand{\mice}{m_{\mbox{\tiny{ice}}}}
\newcommand{\rocean}{\rho_{\mbox{\tiny{ocean}}}}
\newcommand{\tatm}{\tau_{\mbox{\tiny{atm}}}}
\newcommand{\tocean}{\tau_{\mbox{\tiny{ocean}}}}
\newcommand{\tr}{\mathrm{tr}}
\newcommand{\mbody}{\mbox{m}_{\mbox{\tiny{B}}}}

\newcommand{\rC}{\mathrm{C}}
\newcommand{\rL}{\mathrm{L}}
\newcommand{\rW}{\mathrm{W}}
\newcommand{\rH}{\mathrm{H}}
\newcommand{\rB}{\mathrm{B}}
\newcommand{\rD}{\mathrm{D}}

\newcommand{\AH}{{\mathrm{A}}^{\mathrm{H}}}
\newcommand{\tAH}{\tilde{\mathrm{A}}^{\mathrm{H}}}
\newcommand{\AHm}{{\mathrm{A}}^{\mathrm{H}}_{\mathrm{m}}}
\newcommand{\AHD}{{\mathrm{A}}^{\mathrm{H}}_{\mathrm{D}}}

\newcommand{\tAHm}{\tilde{\mathrm{A}}^{\mathrm{H}}_{\mathrm{m}}}

\title{Rigorous analysis of the interaction problem of Sea Ice with a Rigid Body}

\author{Tim Binz}
\address{Technische Universit\"at Darmstadt\\
	Fachbereich Mathematik\\
	Schlossgartenstrasse 7\\
	64289 Darmstadt, Germany}
\email{binz@mathematik.tu-darmstadt.de}

\author{Felix Brandt}
\address{Technische Universit\"at Darmstadt\\
	Fachbereich Mathematik\\
	Schlossgartenstrasse 7\\
	64289 Darmstadt, Germany}
\email{brandt@mathematik.tu-darmstadt.de}

\author{Matthias Hieber}
\address{Technische Universit\"at Darmstadt\\
	Fachbereich Mathematik\\
	Schlossgartenstrasse 7\\
	64289 Darmstadt, Germany}
\email{hieber@mathematik.tu-darmstadt.de}


\begin{document}
	\begin{abstract}
		Consider the set  of equations modeling the motion of a rigid body enclosed in sea ice. Using Hibler's viscous-plastic model for describing sea ice, it  is shown 
		by a certain  decoupling approach that this system admits a unique, local strong solution within the $\rL^p$-setting.
	\end{abstract}

\maketitle

\section{Introduction}
\label{sec:intro}

It is a classical problem in fluid mechanics to study the movement of rigid or elastic bodies immersed in a fluid, see e.g.\ the works of Galdi \cite{Gal:02}, Hoffmann and Starovoitov \cite{HS:99}, 
Desjardins and Esteban \cite{DE:00,DE:99}, Gunzberger, Lee and Seregin \cite{GLS:00}, Feireisl, Hillairet and Ne\v{c}asov\'{a} \cite{FHN:08}, Cumsille, Takahashi and Tucsnak
\cite{CT:08,CT:06}, Geissert, G\"otze and Hieber \cite{GGH:13}, Maity, Raymond, Roy and Vanninathan \cite{RV:14,MRR:20}, and the recent work of Ervedoza, Maity and Tucsnak  \cite{EMT:22}. We also refer to the survey article \cite{GN:18} by Galdi and Neustupa 
for the stationary case. The very recent article \cite{EMT:22} by Ervedoza, Maity and Tucsnak discusses the long-time behaviour of a system accounting for the motion of a rigid body enclosed in a 
viscous incompressible fluid. 

Mathematically, interaction problems of rigid bodies can be described by a moving domain problem coupling a PDE with an ODE. The PDE comes from the geophysical equation which models the surrounding material, whereas the ODE is given by the balance equations for the momentum and angular momentum of the immersed rigid body.
In this article, we investigate the problem of interaction of a rigid body in sea ice. Sea ice as a material exhibits a complex mechanical and thermodynamical behaviour.
A composite of pure ice, liquid brine, air pockets and solid salt is formed by freezing sea water.
As indicated by Feltham \cite{Fel:08} or by Golden \cite{Gol:15}, the details of this formation depend on the laminar or turbulent environmental conditions.
The response of this composite to heating, pressure or mechanical forces is for example different from the response of the (salt-free) glacial ice of ice sheets.
For a recent survey in the Notices of the AMS, see \cite{survey:20}. 

We note that there  is significant interest in understanding the interaction of ice and rigid structures. One particular application is e.g.\ the behaviour of a ship in sea ice.
For numerical simulations of the movements of ships in an ice floe field, we refer e.g.\ to the work of Zhan, Agar, He, Spenced and Molyneux \cite{ZAHSM:10} and Kim, Sawamura \cite{KS:16}; the article  \cite{TP:18} of Tuhkuri and Poloj\"{o}rvi provides a review of ice-structure 
interaction simulations.

W.D.~Hibler suggested in 1979 the governing equations of large-scale sea ice dynamics in a seminal article \cite{Hib:79}. These equations form the basis of many sea ice models in climate science.
Sea ice is here modeled as a material with a very specific constitutive law combined with viscous-plastic rheology. 

During the last decades, various communities have been investigating this set of equations numerically, see e.g.\ \cite{mehalle:21,MK:21,MR:17,SK:18,DWT:15,KDL:15,LT:09}.
Unlike the equations describing atmospheric or oceanic dynamics as e.g.\ the primitive equations, see the seminal article by Cao and Titi \cite{CT:07}, rigorous analysis of the sea ice equations started only very recently by the works of 
Brandt, Disser, Haller-Dintelmann and Hieber \cite{BDHH:22} as well as Liu, Thomas and Titi \cite{LTT:22}. 
The underlying set of equations is a coupled degenerate quasilinear parabolic-hyperbolic system, whose analysis is delicate.

In \cite{BDHH:22}, it is shown by means of the theory of quasilinear evolution equations that a suitable regularization of Hibler's model coupling velocity, thickness and compactness of sea ice is 
locally strongly well-posed and also globally strongly well-posed for initial data close to constant equilibria. 
The approach developed  in \cite{LTT:22} emphasizes the parabolic-hyperbolic character of Hibler's model and proves also local strong well-posedness by means of a different regularization and 
by energy estimates.

It is the aim of this article to present for the first time a rigorous analytical study of the interaction of sea ice with rigid structures.  
We develop an $\rL^p$-theory for strong solutions to the interaction problem of a rigid body trapped in sea ice in a bounded domain. Our main theorem shows local-in-time existence and uniqueness of the interaction problem of sea ice with a rigid body.
Note that sea ice cannot be viewed as a generalized Newtonian fluid in the scope of \cite{GGH:13} or a compressible fluid as in \cite{HM:15}.
The stress tensor, see \eqref{eq:stress tensor} or \eqref{eq:regularizedstresstensor} for its regularized version, is of {\em quasilinear nature}, and it is not possible to express the factors in front of the deformation tensor $\eps$ and its trace $\tr(\eps)$ as a function of the Hilbert-Schmidt norm of the symmetric part of the gradient.
In addition, the ice strength is a function of two unknowns here, namely of the mean ice thickness $h$ and the ice compactness $a$, see \eqref{eq:ice pressure} below.
The linear theory as e.g.\ the maximal $\rL^p$-regularity or the bounded $\Hinfty$-calculus for the associated linearized operator, also referred to as the {\em Hibler operator}, was established only very recently in the work of Brandt, Disser, Haller-Dintelmann and Hieber \cite{BDHH:22}.

The strategy of the proof of our main result is as follows.
Since the domain of the sea ice equation is one of the unknowns, we first transform the problem to a fixed domain.
Our approach in this article is ``monolythic'', i.e., we deal with the entire transformed problem which is still coupled and analyze the linearized operator associated to the coupled sea ice-rigid body system.
Using the results from \cite[Sections~4, 5 and 6]{BDHH:22}, we establish maximal regularity of this linearized operator by means of a decoupling technique.
In addition, the quasilinear terms in the system of PDEs for the sea ice dynamics are estimated by exploiting the regularity of the coordinate transform in conjunction with the respective estimates in \cite[Section~6]{BDHH:22}, and we estimate the nonlinear terms of the ODE. 
At this stage, the quasilinear nature of the stress tensor in Hibler's sea ice model fully comes into play, and we emphasize that these estimates are novel and rely on a variant of nonlinear complex interpolation result due to Bergh \cite{Ber:84}.
In the spirit of our ``monolythic'' approach, we interpret the entire transformed problem as a non-autonomous quasilinear abstract Cauchy problem.
It allows us to use classical local-in-time existence theory instead of developing a more sophisticated fixed point argument as in \cite{GGH:13}, \cite{HMTT:19} or \cite{HM:15}.
Finally, we conclude existence and uniqueness of the original problem via performing the corresponding backwards coordinate transform. 

The above transform to a fixed domain was first introduced by Inoue and Wakimoto \cite{IW:77} and then used by Conca, Cumsille, San Mart\'in, Scheid, Takahashi and Tucsnak \cite{CSMT:00,Tak:03,TT:04,CT:06,CT:08,SMSTT:08} in the context of incompressible fluids, by Geissert, G\"otze and Hieber \cite{GGH:13} in the situation of 
Newtonian and generalized Newtonian fluids, and by Hieber and Murata \cite{HM:15} for the study of compressible fluids.
In the cited articles, this transform is the starting point to show the existence of a unique, local strong solution, or even a global solution in the case of small data, to the coupled system on 
bounded or unbounded fluid domains and in the situation of two or three space dimensions.

In the context of incompressible fluids, the above fluid-structure interaction problem was investigated by many other authors in the weak and strong setting by means of different types of transformations.
The existence of local or global strong solutions has its roots in the works of Galdi \cite{Gal:02}, Galdi and Silvestre \cite{GS:02}, Hoffmann and Starovoitov \cite{HS:99}, Desjardins and 
Esteban \cite{DE:00,DE:99}, Gunzberger, Lee and Seregin \cite{GLS:00} as well as Feireisl, Hillairet and Ne\v{c}asov\'{a} \cite{FHN:08}. 

An abstract framework for linear fluid-solid interaction problems was introduced by Maity and Tucsnak \cite{MT:17}.
Viewing the linearized fluid-solid interaction problems as boundary controlled fluid systems with dynamic boundary feedback and invoking the theory of extrapolation spaces to deal with the boundary conditions, they find that the interaction system can be regarded as a perturbation of the pure fluid equations.
Consequently, they obtain maximal regularity of the linearized fluid-solid system by verifying $\mathcal{R}$-sectoriality by means of a perturbation argument.

Our approach is inspired by the work of Casarino, Engel, Nagel and Nickel \cite{CENN:03} for the question of the generation of a $\rC_0$-semigroup.
In our case, we need maximal regularity, and we use that maximal regularity is preserved under similarity transforms.
In particular, the trace space for the initial data incorporates the boundary and coupling conditions in our situation.
It seems to us that this approach is quicker in our instance.
For fluid-structure problems, both approaches work.
However, in a general abstract framework, our approach allows to deal with a larger class of unbounded feedback operators.

The technique developed in \cite{MT:17} was later also used by  Haak, Maity, Takahashi and Tucsnak \cite{HMTT:19} to show that the interaction of a compressible Navier-Stokes-Fourier fluid with a rigid body is locally strongly well-posed and also globally strongly well-posed for small initial data and by Maity, San Mart\'{i}n, Takahashi and Tucsnak \cite{MSTT:19} in the context of a rigid structure floating in a viscous fluid.

This article is organized as follows.
In \autoref{sec:coupled model and main result}, we briefly recall Hibler's sea ice equations and introduce the interaction problem rigorously.
Afterwards, we give a precise definition of a solution and formulate our main theorem. 
\autoref{sec:coordinate transform} is then devoted to the presentation of the coordinate transform and the computation of the transformed terms. 
In \autoref{sec:auxiliary results}, we provide two auxiliary results concerning non-autonomous quasilinear evolution equations and nonlinear complex interpolation, respectively.
In \autoref{sec:proof main result}, we reformulate the transformed problem as a non-autonomous quasilinear abstract Cauchy problem. Further, we verify the assumptions of the local existence theorem: 
We show maximal regularity of the transformed operator by a 
decoupling argument and establish the Lipschitz estimates of the operator matrix and the right-hand sides.
This results in the proof of the main theorem.

\section{Sea ice interaction with a rigid body and main result}
\label{sec:coupled model and main result}

The present sea ice model is a 2D model, and we consider a bounded domain $\cO \subset \R^2$ of class $\rC^2$.
Moreover, $0<T\le \infty$ represents a positive time.
We denote the time-dependent bounded domain occupied by the rigid body at time $t \in [0,T]$ by $\cB(t)$ and the remaining part of the domain filled by sea ice by $\cD(t) = \cO \setminus \overline{\cB(t)}$.
Moreover, the interface between the body and the sea ice is denoted by $\Gamma(t)$.
To simplify the notation, we introduce $\cD := \cD(0)$, $\cB := \cB(0)$ as well as $\Gamma := \Gamma(0)$.
The outer normal at $\Gamma(t)$ is denoted by $n(t)$, and the sets $\cQD$ and $\cQG$ consist of all points in spacetime where the spatial component is in $\cD(t)$ and $\Gamma(t)$, respectively, i.e.,
\begin{equation*} 
\cQD := \{(t,x) \in \R^3 : t \in (0,T), \enspace x \in \cD(t)\} \qquad \mbox{ and } \qquad \cQG := \{(t,x) \in \R^3 : t \in (0,T), \enspace x \in \Gamma(t)\}.
\end{equation*} 
Note that this can be written as 
\begin{align*}
	\cQD = \dot{\bigcup}_{t \in (0,T)} \{t\} \times \cD(t) \qquad
	\mbox{ and } \qquad 
	\cQG = \dot{\bigcup}_{t \in (0,T)} \{t\} \times \Gamma(t) ,
\end{align*}
where $\dot{\bigcup}$ denotes the disjoint union.
We also use the notation $\cQpO = \{(t,x) \in \R^3 : t \in (0,T), \enspace x \in \partial \cO\}$ in the sequel.

\subsection{Hibler's sea ice model}
\label{ssec:hibler}
\

By $u \colon \cQD \to \R^2$, we denote the horizontal velocity of the sea ice, while $h \colon \cQD \to \R_{+}$ and $a \colon \cQD \to [0,1]$ represent the mean ice thickness and the ice compactness, respectively.
Concerning the mean ice thickness, we impose the constraint that $h > \kappa$ for some small parameter $\kappa > 0$ which indicates the transition to open water. 
More precisely, a value of $h(t,x)$ less than $\kappa$ means that at $(t,x)$, there is open water.
In addition, we suppose that $a \in (\alpha,1-\alpha)$ holds for some small $\alpha > 0$.

Following \cite{Hib:79}, the constitutive law for the ice stress is given by
\begin{equation}\label{eq:stress tensor}
\sigma(u,h,a) = \frac{1}{e^2} \frac{P(h,a)}{\tri(\eps)} \eps + \left(1 - \frac{1}{e^2}\right) \frac{P(h,a)}{2 \tri(\eps)} \tr(\eps) \Itwo - \frac{P(h,a)}{2}\Itwo,
\end{equation}
where $\eps = \eps^{(u)} = \eps(u) = \frac{1}{2}\left(\nabla u + (\nabla u)^\T\right)$ is the deformation tensor, $\Itwo$ denotes the unit matrix in $\R^{2 \times 2}$, $e > 1$ is the ratio of major to minor axes of the elliptical yield curve on which the principal components of the stress lie, and $P$ represents the ice pressure and is defined by
\begin{equation}\label{eq:ice pressure}
P = P(h,a) = p^* h \mathrm{e}^{-c(1-a)},
\end{equation}
for given constants $p^* > 0$ and $c>0$.
In addition, we have
\begin{equation*} 
\triangle^2(\eps) := (\eps_{11}^2 + \eps_{22}^2)\left(1+\frac{1}{e^2}\right) + \frac{4}{e^2}\eps_{12}^2 + 2 \eps_{11} \eps_{22}\left(1-\frac{1}{e^2}\right).
\end{equation*} 

Even though the above law describes an idealized viscous-plastic material, its viscosities become singular if $\triangle$ tends to zero.
Following \cite{BDHH:22} and \cite{MK:21}, see also \cite{KHLFG:00}, we consider for $\delta > 0$ the regularization $\triangle_\delta(\eps) := \sqrt{\delta + \triangle^2(\eps)}$.
Thus, we define the regularized internal ice stress by
\begin{equation}\label{eq:regularizedstresstensor}
	\sigd(u,h,a) := S_\delta(u,h,a) - \frac{P(h,a)}{2} \Itwo := \frac{1}{e^2} \frac{P(h,a)}{\trid(\eps)} \eps + \left(1 - \frac{1}{e^2}\right) \frac{P(h,a)}{2 \trid(\eps)} \tr(\eps) \Itwo - \frac{P(h,a)}{2}\Itwo,
\end{equation}
where $S_\delta$ takes the shape
\begin{equation}\label{eq:Sdelta}
	S_\delta(u,P) := S_\delta (\eps,P) := \frac{P}{2}\frac{\mathbb{S} \eps}{\triangle_\delta (\eps)}.
\end{equation}
Here $\mathbb{S} \colon \R^{2\times 2} \to \R^{2\times 2}$ is a map satisfying
\begin{equation*}
	\mathbb{S} \eps = \begin{pmatrix}
		(1 + \frac{1}{e^2}) \eps_{11} + (1 - \frac{1}{e^2}) \eps_{22} & \frac{1}{e^2} (\eps_{12} + \eps_{21}) \\
		\frac{1}{e^2} (\eps_{12} + \eps_{21}) & (1 - \frac{1}{e^2}) \eps_{11} + (1 + \frac{1}{e^2}) \eps_{22}
	\end{pmatrix}
\end{equation*}
for the deformation tensor $\eps = \eps(u)$. 
Identifying $\eps \in \R^{2 \times 2}$ with the vector $(\eps_{11},\eps_{12},\eps_{21},\eps_{22})^\T$, we find that the map $\mathbb{S}$ corresponds to the positive semi-definite matrix
\begin{equation*}
    \mathbb{S} = \left(\mathbb{S}_{ij}^{kl}\right) = \begin{pmatrix}
    1 + \frac{1}{e^2} & 0 & 0 & 1 - \frac{1}{e^2}\\
    0 & \frac{1}{e^2} & \frac{1}{e^2} & 0\\
    0 & \frac{1}{e^2} & \frac{1}{e^2} & 0\\
    1 - \frac{1}{e^2} & 0 & 0 & 1 + \frac{1}{e^2}
    \end{pmatrix}.
\end{equation*}
Simple comparisons reveal that $\mathbb{S}$ exhibits the symmetries
\begin{equation*}
    \mathbb{S}_{ij}^{kl} = \mathbb{S}_{ji}^{lk} = \mathbb{S}_{kj}^{il} = \mathbb{S}_{kl}^{ij} = \mathbb{S}_{il}^{kj}.
\end{equation*}

With $S_\delta (u,P)$ as introduced in \eqref{eq:Sdelta}, Hibler's operator is defined by
\begin{equation*} 
\AH u := - \div S_\delta (u,P).
\end{equation*} 
Following the calculations in \cite[Section~3]{BDHH:22}, we find that Hibler's operator can be written as an operator in non-divergence form 
\begin{equation}\label{eq:Hibler op nondiv form}
\begin{aligned}
    (\AH u)_i &= -\sum \limits_{j,k,l=1}^2 \frac{P}{2}\frac{1}{\triangle_\delta (\eps)}\left(\mathbb{S}_{ij}^{kl} - \frac{1}{\triangle_\delta ^2 (\eps)}(\mathbb{S} \eps)_{ik} (\mathbb{S} \eps)_{jl} \right) \partial_k \eps_{jl} - 
	\frac{1}{2 \triangle_\delta (\eps)}\sum \limits_{j=1}^2 (\partial_j P) (\mathbb{S} \eps)_{ij}\\
	&= \sum \limits_{j,k,l=1}^2 \frac{P}{2}\frac{1}{\triangle_\delta (\eps)}\left(\mathbb{S}_{ij}^{kl} - \frac{1}{\triangle_\delta ^2 (\eps)}(\mathbb{S} \eps)_{ik} (\mathbb{S} \eps)_{jl} \right) D_k D_l u_j - 
	\frac{1}{2 \triangle_\delta (\eps)}\sum \limits_{j=1}^2 (\partial_j P) (\mathbb{S} \eps)_{ij}
\end{aligned}
\end{equation}
for $i=1,2$ and $D_m = - \mathrm{i} \partial_m$.

Let us now introduce the external forces acting on the sea ice. 
For $\mice = \rice h$ denoting the mass of the sea ice, where $\rice > 0$ is the density, the Coriolis force term is given by $\mice \ccor n \times u$ with Coriolis parameter $\ccor > 0$ and unit vector $n \colon \R^2 \to \R^3$ normal to the surface.
Further, $\mice g \nabla H$ represents the force due to changing sea surface tilt with sea surface dynamic height $H \colon (0,T) \times \R^2 \to [0,\infty)$ and gravity $g$.
The atmospheric wind and oceanic forces are described by the terms $\tatm$ and $\tocean(u)$, respectively, and they take the shape
\begin{equation*} 
\tatm = \ratm \Catm \vert \Uatm \vert \Ratm \Uatm \qquad \mbox{ and } \qquad \tocean(u) = \rocean \Cocean \vert \Uocean - u \vert \Rocean (\Uocean - u),
\end{equation*} 
where $\Uatm$ and $\Uocean$ are the velocity of the surface winds and current, respectively. 
Furthermore, $\Catm$ and $\Cocean$ denote air and ocean drag coefficients, $\ratm$ and $\rocean$ represent the densities for air and sea water, and $\Ratm$ and $\Rocean$ are rotation matrices acting on wind and current vectors.
Except for the sea ice velocity $u$, the quantities involved in $\tatm$ and $\tocean(u)$ are assumed to be constant in space and time for simplicity.
In particular, $\tatm$ and $\tocean(u)$ are independent of the mean ice thickness $h$ and the ice compactness $a$.

In the sequel, we use $f_1$ to denote the external force terms in the momentum equation, so
\begin{equation}\label{eq:rhsseaice}
	f_1(u,h) = - \mice \ccor n \times u - \mice g \nabla H + \tatm + \tocean(u).
\end{equation}

For $f \in \rC_b^1([0,\infty);\R)$ denoting an arbitrary function describing the ice growth rate, for instance the one proposed by Hibler \cite[Section~3]{Hib:79}, the thermodynamic terms in the balance laws are given by
\begin{align}\label{eq:S_h}
	S_h = f\left(\frac{h}{a}\right)a + (1-a)f(0)
\end{align}
and
\begin{align}\label{eq:S_a}
	S_a = \begin{cases*} \frac{f(0)}{\kappa}(1-a),& if $f(0) > 0$, \\ 0, \quad & if $f(0) < 0$, 
	\end{cases*}
	\quad + \quad \begin{cases*} 0,& if $S_h > 0$, \\ \frac{a}{2 h}S_h, & if $S_h < 0$.
	\end{cases*}    
\end{align}

Recalling the shape of the regularized stress tensor $\sigd$ from \eqref{eq:regularizedstresstensor}, the external force term $f_1$ from \eqref{eq:rhsseaice} and the thermodynamic terms $S_h$ and $S_a$ from \eqref{eq:S_h} and \eqref{eq:S_a}, respectively, we obtain that the system of equations accounting for the sea ice dynamics is given by
\begin{equation}\label{eq:seaice}
	\left\{ \begin{array}{rll}
		\mice(u_t + u \cdot \nabla u) &= \div\sigd(u,h,a) + f_1(u,h), & \; \mbox{ in } \cQD,\\[2mm]
		h_t + \div(u h) &= d_h \Delta h + S_h(h,a), & \; \mbox{ in } \cQD, \\[2mm]
		a_t + \div(u a) &= d_a \Delta a + S_a(h,a), & \; \mbox{ in } \cQD, \\[2mm]
		u(0) = u_0, \enspace h(0) &= h_0, \enspace a(0) = a_0,  & \; \mbox{ in } \cD,
	\end{array} \right.
\end{equation}
where $\Delta$ denotes the Laplacian, and $d_h > 0$ as well as $d_a > 0$ represent constants.

\subsection{Sea ice interaction with a rigid body}
\label{ssec:model}
\

For simplicity, we use $v=(u,h,a)$ to denote the principle variable of the sea ice equations in the sequel.
Moreover, $\eta \colon (0,T) \to \R^2$ represents the translational velocity, while $\omega \colon (0,T) \to \R$ describes the angular velocity.
With the regularized stress tensor $\sigd$ as in \eqref{eq:regularizedstresstensor}, the balance equations for the momentum and the angular momentum of the rigid body are
\begin{equation}\label{eq:rigidbody}
	\left\{ \begin{array}{rll}
		\mbody \eta'(t) + \int_{\Gamma(t)} \sigd(v)(t,x) n (t,x) \,\mathrm{d} S &= F(t), & \; t \in (0,T),\\[2mm]
		J \omega'(t) + \int_{\Gamma(t)} (x - x_c(t))^\perp \sigd(v)(t,x) n(t,x) \,\mathrm{d} S &= N(t), & \; t \in (0,T), \\[2mm]
		\eta(0) = \eta_0, \enspace \omega(0) &= \omega_0. \\[2mm]
	\end{array} \right.
\end{equation}
The constants $\mbody$ and $J$ represent the body's mass and inertia tensor. As in \cite{CT:08}, the inertia tensor is given by
\begin{equation*}
	J = \int_{\cB(t)} \rho_\cB \vert x - x_c(t) \vert ^2 \,\mathrm{d} x = \int_{\cB(0)} \rho_\cB \vert y \vert ^2 \,\mathrm{d} y.
\end{equation*}
In particular, $J$ is not time-dependent and $(J \omega)'(t) = J \omega'(t)$. 
The functions $F \colon (0,T) \to \R^2$ and $N \colon (0,T) \to \R$ represent external forces and torques.
Further, $x_c$ describes the position of the body's center of gravity, where we suppose that $x_c(0) = 0$ for convenience. 
We observe that $x_c'(t) = \eta(t)$ for $t \in (0,T)$, so $x_c$ can be deduced from $\eta$ by 
\begin{equation}\label{eq:recovering xc}
    x_c(t) = \int_0^t \eta(s) \, \mathrm{d} s.
\end{equation}

For $y = (y_1,y_2)^\T \in \R^2$, we write $y^\perp = (y_2,-y_1)^\T$. The full velocity of the rigid body is then given by
\begin{equation} 
u_{\cB}(t,x) := \eta(t) + \omega(t) (x - x_c(t))^\perp.
\label{eq:uB}
\end{equation} 
The velocity of the rigid body $u_{\cB}$ as given in \eqref{eq:uB} coincides with the velocity of the sea ice $u$ on their interface $\Gamma(t)$ for every time $t \in (0,T)$.
This equation couples the system of PDEs of the sea ice dynamics with the ODEs accounting for the motion of the rigid body. 

The coupled system of the sea ice dynamics as in \eqref{eq:seaice} and the motion of the enclosed rigid body as in \eqref{eq:rigidbody} is completed by boundary conditions for the unknowns $u, h, a, \eta$ and $\omega$ to the following set of equations
\begin{equation}\label{eq:complete system}
	\left\{ \begin{array}{rll}
		\mice(u_t + u \cdot \nabla u) &= \div\sigd(v) + f_1(u,h), & \; \mbox{ in } \cQD,\\[2mm]
		h_t + \div(u h) &= d_h \Delta h + S_h(h,a), & \; \mbox{ in } \cQD, \\[2mm]
		a_t + \div(u a) &= d_a \Delta a + S_a(h,a), & \; \mbox{ in } \cQD, \\[2mm]
		u(t,x) &= \eta(t) + \omega(t) (x - x_c(t))^\perp, & \; \mbox{ on } \cQG, \\[2mm]
		\partial_\nu h = \partial_\nu a &= 0, & \; \mbox{ on } \cQG, \\[2mm]
		u = 0, \enspace \partial_\nu h = \partial_\nu a &= 0, & \; \mbox{ on } \cQpO, \\[2mm]
		u(0) = u_0, \enspace h(0) = h_0, \enspace a(0) &= a_0, & \; \mbox{ in } \cD, \\[2mm]
		\mbody \eta'(t) + \int_{\Gamma(t)} \sigd(v)(t,x) n (t,x) \,\mathrm{d} S &= F(t), & \; t \in (0,T),\\[2mm]
		J \omega'(t) + \int_{\Gamma(t)} (x - x_c(t))^\perp \sigd(v)(t,x) n(t,x) \,\mathrm{d} S &= N(t), & \; t \in (0,T), \\[2mm]
		\eta(0) = \eta_0, \enspace	\omega(0) &= \omega_0. \\[2mm]
	\end{array} \right.
\end{equation}

\subsection{Main result}
\label{ssec:main}
\

In the sequel, we need the following function spaces. 
For an arbitrary open set $O \subset \R^n$, $n \in \N$, $p,q \in (1,\infty)$, $k \in \N$ and $s \ge 0$, we denote by $\rL^q(O)$, $\rC(O)$, $\rC^k(O)$, $\rC^\infty(O)$, $\rW^{s,q}(O)$, $\rH^{s,q}(O)$ and $\rB^{s}_{qp}(O)$ the spaces of {$q$-integrable} functions, continuous functions, $k$-times continuously differentiable functions, smooth functions, (fractional) Sobolev spaces, Bessel potential spaces and Besov spaces, respectively. 
The definitions are standard and can e.g.\ be found in \cite{Tri:78}. 

To make precise the concept of a solution of the system \eqref{eq:complete system}, we need the definition of function spaces on time-dependent domains $\cD(t)$. 
Let $Z \colon \cQD \to \cD$ be a map such that
\begin{equation*}
    \varphi \colon \cQD \to (0,T) \times \cD, \enspace (t,x) \mapsto (t,Z(t,x))
\end{equation*}
is a $\rC^1$-diffeomorphism and $Z(\tau,\cdot) \colon \cD(\tau) \to \cD$ are $\rC^2$-diffeomorphisms for all $\tau \in [0,T]$. 
For any $p,q \in (1,\infty)$, $s \in \{0,1\}$ and $l \in \{0,1,2\}$, we define then
\begin{equation*}
    \rW^{s,p}(0,T;\rW^{l,q}(\cD(\cdot))) := \{f(t,\cdot) \colon \cD(t) \to \R : f \circ \varphi \in \rW^{s,p}(0,T;\rW^{l,q}(\cD))\}.
\end{equation*}
We remark that $\cQD$ is a $\rC^{1,2}$-manifold with spatial boundary $\cQG \cup \cQpO$ where the only chart is $\varphi \colon \cQD \to (0,T) \times \cD$. 
Our definition of time-dependent function spaces coincides with the well known definition of function spaces on a manifold with boundary. 

Next, we introduce the set of suitable initial data for our problem. 
For $\kappa> 0, \alpha > 0$ sufficiently small as introduced at the beginning of \autoref{ssec:hibler}, consider the set
\begin{equation}\label{eq:V}
    \begin{aligned}
        V := \{&(u_0,h_0,a_0,\eta_0,\omega_0) \in \rB_{qp}^{2-\nicefrac{2}{p}}(\cD;\R^2) \times \rB_{qp}^{2-\nicefrac{2}{p}}(\cD)^2 \times \R^2 \times \R : h_0 > \kappa \text{ and } a_0 \in (\alpha,1-\alpha),\\
	    &\partial_\nu h_0 = \partial_\nu a_0 = 0 \text{ on } \partial \cD, \enspace u_0 = 0 \text{ on } \partial \cO \text{ and } u_0 = \eta_0 + \omega_0 x^\perp \text{ on } \Gamma\}.
    \end{aligned}
\end{equation}

Since $\kappa >0$ and $\alpha>0$ are fixed throughout the article, we omit them in the notation of $V$.
We also observe that we consider $p,q \in (1,\infty)$ such that
\begin{equation}\label{eq:condition p and q}
    \frac{2}{p} + \frac{3}{q} < 1
\end{equation}
in the sequel.
As a consequence, it is especially valid that $\nicefrac{2}{p} + \nicefrac{1}{q} < 1$, or, equivalently, $2 - \nicefrac{2}{p} > 1 + \nicefrac{1}{q}$, so it is legit to consider Dirichlet or Neumann boundary conditions for elements in $\rB_{qp}^{2-\nicefrac{2}{p}}(\cD)$, see e.g.\ \cite[Section~5]{Ama:93}.
On the other hand, as \eqref{eq:condition p and q} implies in particular that $\nicefrac{2}{p} + \nicefrac{2}{q} < 1$, which is in turn equivalent to $2 - \nicefrac{2}{p} > 1 + \nicefrac{2}{q}$, we obtain the embedding
\begin{equation}\label{eq:emb Besov C1}
    \rB_{qp}^{2-\nicefrac{2}{p}}(\cD) \hra \rC^1(\overline{\cD})
\end{equation}
from \cite[Theorem~4.6.1]{Tri:78}, so it is justified to impose the pointwise conditions on $h_0$ and $a_0$ in the definition of $V$.

Now, we are in the position to formulate our main result on local-in-time existence and uniqueness of the coupled sea ice-rigid body system \eqref{eq:complete system}. 

\begin{thm}\label{thm:main theorem}
	Let $p,q \in (1,\infty)$ be such that \eqref{eq:condition p and q} is satisfied, let $\cO \subset \R^2$ be a bounded domain of class $\rC^2$, and consider the domains of the rigid body and the fluid at time zero, $\cB$ and $\cD$, respectively.
	Moreover, let $w_0 = (u_0,h_0,a_0,\eta_0,\omega_0) \in V$, where $V$ is defined precisely in \eqref{eq:V}, 
	and suppose that $F \in \rL^p(0,T;\R^2)$ as well as $N \in \rL^p(0,T)$.
	If for some $d>0$ it holds that $\dist(\cB,\partial \cO) > d$, then there exists $T' \in (0,T]$ and a map $Z \in \rC^1([0,T'];\rC^2(\R^2))$ such that $Z(\tau,\cdot) \colon \cD(\tau) \to \cD$ are $\rC^2$-diffeomorphisms for all $\tau \in [0,T']$, and \eqref{eq:complete system} admits a unique solution $(u,h,a,\eta,\omega)$ such that
	\begin{equation*}
		\begin{aligned}
			u &\in \rW^{1,p}(0,T';\rL^q(\cD(\cdot);\R^2)) \cap \rL^{p}(0,T';\rW^{2,q}(\cD(\cdot);\R^2)), \\
			h &\in \rW^{1,p}(0,T';\rL^q(\cD(\cdot))) \cap \rL^{p}(0,T';\rW^{2,q}(\cD(\cdot))), \\
			a &\in \rW^{1,p}(0,T';\rL^q(\cD(\cdot))) \cap \rL^{p}(0,T';\rW^{2,q}(\cD(\cdot))), \\
			\eta &\in \rW^{1,p}(0,T';\R^2), \quad
			\omega \in \rW^{1,p}(0,T') .
		\end{aligned}
	\end{equation*}
\end{thm}

\begin{rmk}
A solution $(u,h,a,\eta,\omega)$ in the regularity class as in \autoref{thm:main theorem} is called a strong solution.
\end{rmk}

\section{Coordinate transform}
\label{sec:coordinate transform}

We present the diffeomorphism accounting for the transform from the moving on the fixed domain in this section, and we also compute the transformed system of equations. 
It is important to note that the coordinate  transform is an unknown part of the solution of our system.
Throughout this section, we consider a fixed pair $(\eta,\omega) \in \rW^{1,p}(0,T;\R^2) \times \rW^{1,p}(0,T)$.
Now, for the matrix
\begin{equation*}
    m(t) = \omega(t) \begin{pmatrix}
		0 & 1\\ -1 & 0
	\end{pmatrix},
\end{equation*}
satisfying $m(t) x = \omega(t) x^\perp$, we take the differential equation
\begin{equation}\label{eq:ODE Z_0}
	\left\{ \begin{array}{rll}
		\partial_t Z_0(t,y) &= m(t) (Z_0(t,y) - x_c(t)) + \eta(t), & \; (0,T) \times \R^2,\\[2mm]
		Z_0(0,y) &= y, & \; y \in \R^2, \\[2mm]
	\end{array} \right.
\end{equation}
into account. 
As discussed in \autoref{ssec:model}, the coordinates are chosen such that the center of gravity of the rigid body at time $0$ is the origin, i.e., $x_c(0) = 0$. 
The corresponding solution then takes the shape $Z_0(t,y) = Q(t)y + x_c(t)$, where $Q(t) \in \mathrm{SO}(2)$ and $Q \in \rW^{2,p}(0,T;\R^{2 \times 2})$ provided $(\eta,\omega) \in \rW^{1,p}(0,T;\R^2) \times \rW^{1,p}(0,T)$. 
It follows from \eqref{eq:recovering xc} and \eqref{eq:ODE Z_0} that $Q$ is the unique solution of 
\begin{equation}\label{eq:Q}
	\left\{ \begin{array}{rll}
		\partial_t Q(t) &= m(t) Q(t), & \; t \in (0,T),\\[2mm]
		Q(0) &= \Id. & \\[2mm]
	\end{array} \right.
\end{equation}
Further, the inverse $Y_0(t)$ of $Z_0(t)$ is given by
\begin{equation*}
    Y_0(t,x) = Q^\T(t)(x - x_c(t)),
\end{equation*}
and it satisfies the differential equation
\begin{equation*}
	\left\{ \begin{array}{rll}
		\partial_t Y_0(t,x) &= - \tm(t) Y_0(t,x) - \xi(t), & \; (0,T) \times \R^2,\\[2mm]
		Y_0(0,x) &= x, & \; x \in \R^2,
	\end{array} \right.
\end{equation*}
with
\begin{equation*}
    \tm(t) := Q^\T(t) m(t) Q(t), \enspace \xi(t) := Q^\T(t) \eta(t).
\end{equation*}

The next step is to modify the diffeomorphisms $Z_0, Y_0$ of $\cD(t)$ and $\cB(t)$ such that they rotate space only in an appropriate open neighborhood of the rotating and translating body, and they must not rotate or translate the outer boundary $\partial \cO$.
Here we follow the strategy of Geissert, G\"otze and Hieber for their consideration of a bounded fluid domain, see \cite[Section~3 and Section~7]{GGH:13}.
However, it is not necessary to include a Bogovski\u{\i} operator in our case, as we are not in the situation of an incompressible fluid.

The new diffeomorphism is now defined implicitly, using an ODE of the form \eqref{eq:ODE Z_0}, namely
\begin{equation}\label{eq:ODE diffeo Z}
	\left\{ \begin{array}{rll}
		\partial_t Z(t,y) &= b(t,Z(t,y)), & \; (0,T) \times \R^2,\\[2mm]
		Z(0,y) &= y, & \; y \in \R^2. \\[2mm]
	\end{array} \right.
\end{equation}
The right-hand side $b$ in \eqref{eq:ODE diffeo Z} determines the modified velocity of the change of coordinates. 
Close to the rigid body, $b$ should be equal to the velocity of the body, while it is supposed to be zero further away. 
Additionally, considering that the rigid body starts from a position with some distance from the boundary of the sea ice domain and moves with a continuous velocity, we restrict the solution to a time that guarantees that a small distance remains.
More precisely, in the statement of \autoref{thm:main theorem}, we assume that $\dist(\cB,\partial \cO) > d$, and we set $b$ such that a distance of $\frac{d}{2}$ between the body and the outer boundary is maintained.
It is also important that $b$ is smooth in the space variables. 
To obtain these properties, we define a cut-off function $\chi \in \rC^\infty(\R^2;[0,1])$ by
\begin{equation*}
    \chi(y) := \begin{cases*} 1,& if $\dist(y,\partial \cO) \ge d$, \\ 0, \quad & if $\dist(y,\partial \cO) \le \frac{d}{2}$,
    \end{cases*}
\end{equation*}
and a time-dependent vector field $b \colon [0,T] \times \R^2 \to \R^2$ by
\begin{equation}\label{eq:shape of b}
    b(t,x) := \chi(x - x_c(t))[m(t)(x-x_c(t)) + \eta(t)].
\end{equation}
We observe that $b \in \rW^{1,p}(0,T;\rC_c^\infty(\R^2))$ by construction.
As we assume that $\dist(\cB,\partial \cO) > d$, i.e., the body starts with a positive distance from the outer sea ice boundary, and by virtue of $x_c(0) = 0$, we obtain that $\dist(x,\partial \cO) \ge d$ for every $x \in \Gamma$.
Therefore,
\begin{equation*}
    b(0,x) = m(0)(x - x_c(0)) + \eta(0) = \omega_0 x^\perp + \eta_0
\end{equation*}
is valid for $x \in \Gamma$, i.e., $b_{|\Gamma} = \omega_0 x^\perp + \eta_0$.

For $(\eta, \omega) \in \rW^{1,p}(0,T;\R^2) \times \rW^{1,p}(0,T)$ the Picard-Lindel\"of theorem implies that the equation \eqref{eq:ODE diffeo Z} admits a unique solution $Z \in \rC^1(0,T;\rC^\infty(\R^2))$. 
Besides, the solution has continuous mixed partial derivatives $\frac{\partial^{\vert \alpha \vert + 1} Z}{\partial t (\partial y_j)^\alpha}$ and $\frac{\partial^{\vert \alpha \vert} Z}{(\partial y_j)^\alpha}$, where $\alpha \in \N_0^3$ is a multi-index.  
Further, the elements of the Jacobi matrix $J_Z$ of the diffeomorphism are of the form
\begin{equation*} 
(J_Z)_{ij}(t,y) = \partial_j Z_i(t,y) = \delta_{ij} + \int_0^t \frac{\partial b_i}{\partial y_j}(s,Z(s,y)) \d s.
\end{equation*} 

A proof of the following lemma can be found in \cite[Section~2]{HM:15}.
\begin{lem}\label{lem:JXinvertible}
	If either $T_0 \in (0,T]$ is small enough or $\| \nabla_y b \|_{\rL^\infty([0,T] \times \R^2;\R^{2 \times 2})} < c$ for some sufficiently small constant $c>0$, then $J_Z(t,\cdot)$ is invertible for every $t \in (0,T_0)$ or even for every $t \in (0,T)$ in the second case. 
\end{lem}

The inverse transform $Y$ of $Z$ satisfies the equation
\begin{equation}\label{eq:transform Y}
	\left\{ \begin{array}{rll}
		\partial_t Y(t,x) &= b^{(Y)}(t,Y(t,x)), & \; (0,T_0) \times \R^2,\\[2mm]
		Y(0,x) &= x, & \; x \in \R^2, \\[2mm]
	\end{array} \right.
\end{equation}
where 
\begin{equation}\label{eq:shape of b(Y)} 
b^{(Y)}(t,y) := - J_Z^{-1}(t,y) b(t,Z(t,y)),
\end{equation} 
which is well-defined for $t < T_0$ in view of \autoref{lem:JXinvertible}. 
It follows immediately that
\begin{equation*}
	J_Z(t,y) J_Y(t,Z(t,y)) = \Id . \label{eq:J_Z J_Y}
\end{equation*}
Note that by this definition, $b^{(Y)}$ and $Y$ possess the same space and time regularity as $b$ and $Z$.
We emphasize that the diffeomorphism $Z$ accounts for the transform from the moving domain to the fixed domain, and we observe that $Z$ and $Y$ coincide with $Z_0$ and $Y_0$ provided the rigid body is sufficiently far away from the boundary of the sea ice domain, while it holds that $\partial_t Z(t,y) = \partial_t Y(t,x) = 0$ if the rigid body comes close to the boundary.

For $(t,y) \in [0,T) \times \R^2$, $Z$ as in \eqref{eq:ODE diffeo Z} and $Q$ as introduced in \eqref{eq:Q}, we now define
\begin{equation*}
    \begin{aligned}
    \tu(t,y) &:= u(t,Z(t,y)),\\
	\thseaice(t,y) &:= h(t,Z(t,y)),\\
	\ta(t,y) &:= a(t,Z(t,y)),\\
	\xi(t) &:= Q^\T(t) \eta(t),\\
	\Omega(t) &:= \omega(t),\\
	\tF(t) &:= Q^\T(t) F(t),\\
	\tN(t) &:= N(t),\\
	\sT_\delta(\tu(t,y),\thseaice(t,y),\ta(t,y)) &:= Q^\T(t) \sigd(\tu(t,y),\thseaice(t,y),\ta(t,y)) Q(t),\\
	I &:= J,\\
	\tn &:= Q^\T(t) n(t).
    \end{aligned}
\end{equation*}
Using that $Q(t) \in \mathrm{SO}(2)$ as well as $\Omega(t) = \omega(t)$, we infer that
\begin{equation} 
\tm(t)x = Q^\T(t) m(t) Q(t)x = \omega(t) \det(Q(t)) \begin{pmatrix} x_2 \\ - x_1 \end{pmatrix} = \Omega(t) x^\perp.
\label{eq:tilde m}
\end{equation}

The inertia tensor remains unchanged, so time-independence is also preserved for $I$. 
We remark that $\tn$ represents the outer normal at $\cB$. 
Moreover, we compute for $v = (u,h,a)$ and $\tv = (\tu,\thseaice,\ta)$ that
\begin{equation*}
    \int_{\Gamma(t)} \sigd(v)n(t) \,\mathrm{d} S = Q \int_{\Gamma} \sT_\delta(\tv) \tn \,\mathrm{d} S \qquad \mbox{ as well as } \qquad \int_{\Gamma(t)} (x - x_c(t))^\perp \sigd(v) n(t) \,\mathrm{d} S = \int_{\Gamma} y^\perp \sT_\delta(\tv) \tn \,\mathrm{d} S,
\end{equation*}
where we made use of $Q(t) \in \mathrm{SO}(2)$ to establish that $y^\perp Q z = (Q^\T y^\perp)z$ holds true for all $y,z \in \R^2$.

Having introduced the transformed translational and angular velocity $\xi$ and $\Omega$, we summarize the procedure to construct $Z$ and $Y$ from given $\xi$ and $\Omega$.

\begin{rmk}\label{rmk:procedure Z and Y from xi and Omega}
    Given $(\xi,\Omega) \in \rW^{1,p}(0,T;\R^2) \times \rW^{1,p}(0,T)$, we first obtain $Q \in \rW^{2,p}(0,T;\R^{2 \times 2})$ by solving 
    \begin{equation*}
        \partial_t Q^\T(t) = \tm(t) Q^\T(t), \quad t \in (0,T), \quad Q^\T(0) = \Id, 
    \end{equation*}
    where $\tm(t)x = \Omega(t) x^\perp$ for $x \in \R^2$ as in  \eqref{eq:tilde m}.
    Next, we derive the original translational and angular body velocities from $\eta(t) = Q^\T(t)\xi(t)$ and $\omega(t) = \Omega(t)$, and we then set, as in \eqref{eq:shape of b}, 
    \begin{equation*}
        b(t,x) = \chi(x - x_c(t))[m(t)(x-x_c(t)) + \eta(t)],
    \end{equation*}
    where $x_c(t)$ is recovered from the aforementioned $\eta(t)$ via $x_c(t) = \int_0^t \eta(s) \,\mathrm{d} s$, see \eqref{eq:recovering xc}.
    Having the right-hand side $b$ at hand, we solve \eqref{eq:ODE diffeo Z} to deduce $Z$.
    For $T_0$ as in \autoref{lem:JXinvertible} and $t \in (0,T_0)$, we set $b^{(Y)}(t,y) = - J_Z^{-1}(t,y) b(t,Z(t,y))$ as in \eqref{eq:shape of b(Y)} and subsequently solve \eqref{eq:transform Y} to obtain the inverse $Y$ of $Z$.
\end{rmk}

The following estimates can be concluded from the procedure described in \autoref{rmk:procedure Z and Y from xi and Omega}, see \cite[Proposition~6.1]{GGH:13} and the lemmas thereafter for a proof.
For simplicity, we use $\| \cdot \|_{\infty,\infty}$ to denote the norm associated to $\rL^\infty(0,T;\rL^\infty(\R^2))$.
In addition, the index $i \in \{1,2\}$ means that the corresponding diffeomorphism is associated to $\xi_i$ and $\Omega_i$.

\begin{prop}\label{prop:dep Z and Y on xi and Omega}
    Let $T>0$, and assume that $(\xi_1,\Omega_1), (\xi_2,\Omega_2) \in \rW^{1,p}(0,T;\R^2) \times \rW^{1,p}(0,T)$.
    For $i \in \{1,2\}$, it then holds that $Z_i, Y_i \in \rC^1(0,T_0;\rC^\infty(\R^2))$, and the estimates
    \begin{equation*}
        \begin{aligned}
            \| \partial^\alpha Z_i \|_{\infty,\infty} + \| \partial^\alpha Y_i \|_{\infty,\infty} &\le C(K_i) \quad \text{as well as}\\
            \| \partial^\beta (Z_1 - Z_2) \|_{\infty,\infty} + \| \partial^\beta (Y_1 - Y_2) \|_{\infty,\infty} &\le C(K_i) T (\|\xi_1 - \xi_2 \|_{\rL^\infty(0,T;\R^2)} + \| \Omega_1 - \Omega_2 \|_{\rL^\infty(0,T)})
        \end{aligned}
    \end{equation*}
    are valid for all multi-indices $\alpha, \beta$ such that $1 \le |\alpha| \le 3$ and $0 \le |\beta| \le 3$.
    The constants only depend on the norms $K_i := \| \xi_i \|_{\rL^\infty(0,T;\R^2)} + \| \Omega_i \|_{\rL^\infty(0,T)}$, but they do not depend on $\xi_i$ or $\Omega_i$ directly.
\end{prop} 

Recalling Hibler's operator $\AH$ from \eqref{eq:Hibler op nondiv form}, we denote the coefficients of the principal part of $\AH$ by 
\begin{equation}
	{a_{ij}^{kl}(\nabla u,P) := \frac{P}{2}\frac{1}{\triangle_\delta (\eps)}\left(\mathbb{S}_{ij}^{kl} - \frac{1}{\triangle_\delta ^2 (\eps)}(\mathbb{S} \eps)_{ik} (\mathbb{S} \eps)_{jl} \right)}.  
	\label{eq:coeff}  
\end{equation}
Then for sufficiently smooth initial data, its linearization at $v_0 = (u_0,h_0,a_0)$ is given by 
\begin{equation}
	[\AH(v_0)u]_i = \sum_{j,k,l=1}^2 a_{ij}^{kl}(\eps(u_0),P(h_0,a_0))D_k D_l u_j  - \frac{1}{2 \triangle_\delta (\eps(u_0))}\sum \limits_{j=1}^2 (\partial_j P(h_0,a_0)) (\mathbb{S} \eps(u))_{ij}.
	\label{eq:Hibler op}
\end{equation} 
We start by calculating the transformed symmetric part of the gradient. 
Using the chain rule and arguing similarly as in \cite[Section~9]{GGH:13}, we obtain
\begin{equation*}
    \begin{aligned}
        2 \eps_{ij}^{(u)}(t,x) 
        &= (\partial_i u_j)(t,x) + (\partial_j u_i)(t,x)\\
        &= \sum_{k=1}^2 (\partial_i Y_k)(t,Z(t,y)) \partial_k \tu_j(t,y) + (\partial_j Y_k)(t,Z(t,y)) \partial_k \tu_i(t,y) =: 2 \tepsilon_{ij}^{(\tu)}(t,y),
    \end{aligned}
\end{equation*}
where $\tepsilon = \tepsilon^{(\tu)} = \tepsilon(\tu)$ denotes the transformed symmetric part of the gradient.

Inserting $\tepsilon(\tu)$ as well as $\thseaice$ and $\ta$ into the Hibler operator from \eqref{eq:Hibler op nondiv form} and computing the transformed derivatives, we calculate that the transformed Hibler operator is given by
\begin{equation}\label{eq:Hibler op transformed}
    \begin{aligned}
        \tAH(t,\tw) \tu
	    &= -\sum_{j,k,l,m=1}^2 a_{ij}^{klm}(\tepsilon(\tu),P(\thseaice,\ta)) \partial_m \tepsilon_{jl}(\tu)\\
	    &\quad \quad - \frac{1}{2 \triangle_\delta(\tepsilon(\tu))} p^\ast \mathrm{e}^{-c(1-\ta)} \sum_{j,k=1}^2 \partial_j Y_k \bigl(\partial_k \thseaice + c \partial_k \ta\bigr) (\S \tepsilon(\tu))_{ij},
    \end{aligned}
\end{equation}
where $\tw = (\tu,\thseaice,\ta,\xi,\Omega)$, $a_{ij}^{klm}(\tepsilon(\tu),P(\thseaice,\ta)) = (\partial_k Y_m) a_{ij}^{kl}(\tepsilon(\tu),P(\thseaice,\ta))$ and
\begin{equation*}
	\partial_m \tepsilon_{jl}(\tu) = \frac{1}{2} \sum_{n=1}^2 \bigl((\partial_m \partial_j Y_n) \partial_n \tu_l + (\partial_j Y_n) \partial_m \partial_n \tu_l + (\partial_m \partial_l Y_n) \partial_n \tu_j + (\partial_l Y_n) \partial_m \partial_n \tu_j\bigr).
\end{equation*}
Note that the right-hand side in \eqref{eq:Hibler op transformed} can be regarded as a function depending on $(t,\tu,\thseaice,\ta,\xi,\Omega) = (t,\tw)$ by virtue of \autoref{rmk:procedure Z and Y from xi and Omega}.
Recalling $P = p^\ast h e^{-c(1-a)}$, we find that the lower order terms corresponding to $\div \frac{P}{2} \Itwo$ transform to
\begin{equation} 
(\tB_1(t,\tw) \thseaice)_i = \frac{p^\ast \mathrm{e}^{-c(1-\ta)}}{2} \sum_{j=1}^2 (\partial_i Y_j) \partial_j \thseaice \qquad \mbox{ and } \qquad 
(\tB_2(t,\tw) \ta)_i = \frac{c p^\ast \thseaice \mathrm{e}^{-c(1-\ta)}}{2} \sum_{j=1}^2 (\partial_i Y_j) \partial_j \ta,
\label{eq:B_1,2}
\end{equation} 
where the dependence on $(t,\tw)$ is again implied by \autoref{rmk:procedure Z and Y from xi and Omega}.
Introducing the metric contravariant tensor 
\begin{equation}\label{eq:gij}
g^{ij} := g^{ij}(t,\xi,\Omega) = \sum_{k=1}^2 (\partial_k Y_i) (\partial_k Y_j),
\end{equation}
we determine the transformed Laplacian operators to be given by
\begin{equation} 
    \tL \thseaice := \tL(t,\xi,\Omega) \thseaice = \sum_{j=1}^2 (\Delta Y_j) \partial_j \thseaice + \sum_{j,k=1}^2 g^{jk} \partial_k \partial_j \thseaice,
\label{eq:transformed Laplacian h,a}
\end{equation} 
and the right-hand side can be regarded as a function of $(t,\tw)$ thanks to \autoref{rmk:procedure Z and Y from xi and Omega}.
The shape of $\tL\ta := \tL(t,\xi,\Omega) \ta$ is completely analogous.

Next, we observe that the force terms $f_1$, $S_h$ and $S_a$ do not contain any derivatives.
Therefore, we get the respective transformed terms by simply inserting the transformed variables, i.e., $f_1(\tu,\thseaice)$, $S_h(\thseaice,\ta)$ and $S_a(\thseaice,\ta)$, where we recall the respective shapes from \eqref{eq:rhsseaice}, \eqref{eq:S_h} and \eqref{eq:S_a}.

We continue by determining the transformed bilinear terms.
To this end, we calculate
\begin{equation}\label{eq:transformed transport terms}
        ((u \cdot \nabla) u)_k
        = \sum_{i,j=1}^2 \tu_i (\partial_i Y_j) \partial_j \tu_k \qquad \mbox{ and } \qquad \div(u h) = \sum_{i,j=1}^2 (\partial_i Y_j) \bigl(\tu_i \partial_j \thseaice + \thseaice \partial_j \tu_i\bigr).
\end{equation}
The transformed time derivative of a function $b$ is given by
\begin{equation}\label{eq:transformed time derivative}
    b_t = \tb_t + \sum_{j=1}^2 \dot{Y_j} \partial_j \tb,
\end{equation}
where $\dot{Y_j}$ denotes the time derivative of $Y_j$.

Using the shapes of the transformed time derivative from \eqref{eq:transformed time derivative} and the transformed transport terms from \eqref{eq:transformed transport terms}, respectively, and recalling for $\tw = (\tu,\thseaice,\ta,\xi,\Omega)$ the transformed Hibler operator $\tAH(t,\tw)$ from \eqref{eq:Hibler op transformed}, the transformed lower order terms $\tB_1(t,\tw)$ and $\tB_2(t,\tw)$ from \eqref{eq:B_1,2}, $f_1$ from \eqref{eq:rhsseaice}, $\tL$ from \eqref{eq:transformed Laplacian h,a}, $S_h$ from \eqref{eq:S_h} as well as $S_a$ from \eqref{eq:S_a}, we obtain that \eqref{eq:complete system} rewrites as the following system on a cylindrical domain $(0,T) \times \cD$ when applying the transform to a fixed domain:
\begin{equation}\label{eq:transformed system}
	\left\{
	\begin{aligned}
		\rice \thseaice \left(\tu_t + \sum_{j=1}^2 \dot{Y_j} \partial_j \tu + \sum_{i,j=1}^2 \tu_i (\partial_i Y_j) \partial_j \tu\right) &= -\tAH(t,\tw)\tu + \tB_1(t,\tw) \thseaice\\
		&\quad + \tB_2(t,\tw) \ta +  f_1(\tu,\thseaice), &&\text{ in } (0,T) \times \cD,\\
		\thseaice_t + \sum_{j=1}^2 \dot{Y_j} \partial_j \thseaice + \sum_{i,j=1}^2 (\partial_i Y_j) \bigl(\tu_i \partial_j \thseaice + \thseaice \partial_j \tu_i\bigr) &= d_h \tL(t,\xi,\Omega) \thseaice + S_h(\thseaice,\ta), &&\text{ in } (0,T) \times \cD,\\
		\ta_t + \sum_{j=1}^2 \dot{Y_j} \partial_j \ta + \sum_{i,j=1}^2 (\partial_i Y_j) \bigl(\tu_i \partial_j \ta + \ta \partial_j \tu_i\bigr) &= d_a \tL(t,\xi,\Omega)\ta +  S_a(\thseaice,\ta), &&\text{ in } (0,T) \times \cD,\\
		\tu(t,y) &= \xi(t) + \Omega(t) y^\perp, &&\text{ on } (0,T) \times \Gamma, \\
		\partial_\nu \thseaice = \partial_\nu \ta &= 0, &&\text{ on } (0,T) \times \Gamma, \\
		\tu = 0, \enspace \partial_\nu \thseaice = \partial_\nu \ta &= 0, &&\text{ on } (0,T) \times \partial \cO, \\
		\tu(0) = u_0, \enspace \thseaice(0) = h_0, \enspace \ta(0) &= a_0, &&\text{ on } \cD, \\
		\mbody \xi'(t) + Q \int_{\Gamma} \mathcal{T}_\delta(\tu,\thseaice,\ta)(t,y) \tn (t,y) \,\mathrm{d} S &= \tF(t) + \mbody \Omega \xi^\perp, &&\, t \in (0,T),\\
		I \Omega'(t) + \int_{\Gamma} y^\perp \mathcal{T}_\delta(\tu,\thseaice,\ta)(t,y) n(t,y) \,\mathrm{d} S &= \tN(t), && \; t \in (0,T), \\
		\xi(0) = \eta_0, \enspace \Omega(0) &= \omega_0. && \\
	\end{aligned}
	\right. 
\end{equation}

Our main result, \autoref{thm:main theorem}, can now be rephrased as follows.

\begin{thm}\label{thm:main thm transformed}
	Let $p,q \in (1,\infty)$ be such that \eqref{eq:condition p and q} is fulfilled, let $\cO \subset \R^2$ be a bounded domain of class $\rC^2$, and consider the domains of the rigid body and the fluid at time zero, $\cB$ and $\cD$, respectively.
	Moreover, let $\tw_0 = (\tu_0,\thseaice_0,\ta_0,\xi_0,\Omega_0) \in V$, where $V$ is defined precisely in \eqref{eq:V}, and suppose that $F \in \rL^p(0,T;\R^2)$ as well as $N \in \rL^p(0,T)$.
	If for some $d>0$ it holds that $\dist(\cB,\partial \cO) > d$, then there exists $T' \in (0,T]$ and a map $Z \in \rC^1(0,T';\rC^2(\R^2))$ such that $Z(\tau,\cdot) \colon \cD(\tau) \to \cD$ are $\rC^2$-diffeomorphisms for all $\tau \in [0,T']$, and \eqref{eq:transformed system} admits a unique strong solution
	\begin{equation*}
		\begin{aligned}
			\tu &\in \rW^{1,p}(0,T';\rL^q(\cD;\R^2)) \cap \rL^{p}(0,T';\rW^{2,q}(\cD;\R^2)), \\
			\thseaice &\in \rW^{1,p}(0,T';\rL^q(\cD)) \cap \rL^{p}(0,T';\rW^{2,q}(\cD)), \\
			\ta &\in \rW^{1,p}(0,T';\rL^q(\cD)) \cap \rL^{p}(0,T';\rW^{2,q}(\cD)), \\
			\xi &\in \rW^{1,p}(0,T';\R^2), \qquad \text{and} \qquad 
			\Omega \in \rW^{1,p}(0,T') .
		\end{aligned}
	\end{equation*}
\end{thm}

\section{Auxiliary results}
\label{sec:auxiliary results}

In this section, we present the tools needed for the proof of our main result, namely the local existence theorem and the variant of a nonlinear complex interpolation result due to Bergh.

\subsection{A local well-posedness result for non-autonomous quasilinear evolution equations}
\label{subsec:quasilin result}

\

Throughout this subsection, we denote by $A$ the quasilinear operator, by $F$ the nonlinear right-hand side and by $u$ the principle variable of the evolution equation.
The non-autonomous quasilinear abstract Cauchy problem takes the shape
\begin{equation}
	\label{eq:quasilinear}
	\left\{
	\begin{aligned} 
		\dot{u} + A(t,u)u &= F(t,u), &&t \in [0,T], \\
		u(0) &= u_0.
	\end{aligned}
	\right. 
\end{equation}

By $X_0$ and $X_1$, we denote the ground space and the regularity space, respectively, and we assume that $X_1 \hra X_0$ is dense, $\F_T:= \rL^p(0,T;X_0)$ is the data space, while $\E_T := \rL^p(0,T;X_1) \cap \rW^{1,p}(0,T;X_0)$ is the maximal regularity space.
Moreover, we denote by $X_\gamma$ the trace space, and it is well known that $X_\gamma = (X_0,X_1)_{1-\nicefrac{1}{p},p}$, see e.g.\ \cite[Proposition~3.4.4]{PS:16}.
The set $V$ represents an open subset of $X_\gamma$.
Additionally, we assume the following structure conditions on the nonlinearities as well as on the linearized operator.

\begin{asu}\label{ass:quasilinear thm}
	(A1) We have $A \in \rC([0,T] \times V,\sL(X_1,X_0))$. 
	Given $u_0 \in V$, there is $R_0 > 0$ such that $\overline{B}_{X_\gamma}(u_0,R_0) \subset V$, and for all $R \in (0,R_0)$, there exists a Lipschitz constant $L(R) > 0$ independent of $\tau$ with
	\begin{equation*}
		\| A(\tau,u_1) v - A(\tau,u_2) v \|_{X_0} \leq L(R) \| u_1 - u_2 \|_{X_\gamma} \cdot \| v \|_{X_1} 
	\end{equation*} 
	for all $\tau \in [0,T]$, $v \in X_1$ and all $u_1,u_2 \in X_\gamma$ with $\| u_i - u_0 \|_{X_\gamma} \leq R$, $i=1,2$.
	
	(A2) For the mapping $F \colon [0,T] \times X_\gamma \to X_0$, we assume 
	\begin{enumerate}[(i)]
		\item $F(\cdot,u)$ is measurable for every $u \in V$,
		\item $F(\tau,\cdot)\in \rC(V,X_0)$ for almost all $\tau \in [0,T]$,
		\item $F(\cdot,u) \in \F_T$ holds for every $u \in V$,
		\item given $u_0 \in V$, for every $R > 0$ such that $\overline{B}_{X_\gamma}(u_0,R) \subset V$, there exists $\varphi_R \in \rL^p(0,T)$ with
		\begin{equation*}
			\| F(\tau,u_1) - F(\tau,u_2) \|_{X_0}
			\leq \varphi_R(\tau) \cdot \| u_1 - u_2 \|_{X_\gamma} 
		\end{equation*}
		for almost all $\tau \in [0,T]$ and all $u_1,u_2 \in X_\gamma$ with $\| u_i - u_0 \|_{X_\gamma} \leq R$, $i=1,2$.
	\end{enumerate}  
	
	(A3) The operator $A(0,u_0)$ has maximal regularity on $X_0$ for every $u_0 \in V$.
\end{asu}

The following proposition yields local existence of a unique strong solution to the above evolution equation \eqref{eq:quasilinear} under the assumptions presented in \autoref{ass:quasilinear thm}.
We remark that a result of this type is well known in the autonomous case, see e.g.\ \cite[Theorem~5.1.1]{PS:16}, and the non-autonomous case also follows by mimicking and slightly adjusting the arguments therein.
We refer to \cite[Section~2]{Pru:03} and \cite[Section~7.1]{Denk:21} for a discussion of the non-autonomous situation.

\begin{prop}\label{prop:quasilinear}
	Let $p \in (1,\infty)$, assume that $u_0 \in V$, and make \autoref{ass:quasilinear thm}. 
	Then there exists $T' \in (0,T]$ such that \eqref{eq:quasilinear} has a unique solution $u \in \E_{T'}$ in $(0,T')$. 
\end{prop}

\subsection{A variant of a nonlinear complex interpolation result}
\label{subsec:bergh variant}

\

This subsection is dedicated to stating a variant of a nonlinear complex interpolation result due to Bergh \cite{Ber:84}, and we first recall the underlying setting.
Let $(E_0,E_1)$ be a couple of complex Banach spaces that are both embedded in a common Hausdorff topological vector space.
By $\Sigma_E$, we denote the sum of the Banach spaces $E_0$ and $E_1$, i.e., $\Sigma_E = E_0 + E_1$, and by $\sF(E_0,E_1)$, we denote the Banach space of all functions $f$ defined on the strip $0 \le \re z \le 1$ in the complex plane such that
\begin{enumerate}[(i)]
	\item $f(z) \in \Sigma_E$ and $f$ is continuous in $\Sigma_E$ on $0 \le \re z \le 1$,
	\item $f(j + \mathrm{i}t) \in E_j$, $j=0,1$, and $f(j + \mathrm{i}\cdot)$ is continuous in $E_j$ with $\lim_{|t| \to \infty} f(j + \mathrm{i}t) = 0$ in $E_j$,
	\item $f$ is analytic in $\Sigma_E$ on $0 < \re z < 1$, \mbox{ and}
	\item $\| f \|_{\sF} := \max_{j=0,1} (\sup_{t \in \R} \| f(j + \mathrm{i}t) \|_{E_j})$.
\end{enumerate}

The complex interpolation space $E_\theta$ is the Banach space of all $f(\theta)$ with $f \in \sF(E_0,E_1)$, where we have chosen $0 < \theta < 1$, and the norm of $a \in E_\theta$ is $\| a \|_\theta := \inf \| f \|_{\sF}$, where $f(\theta) = a$.

A slight modification of the standard arguments in \cite{Ber:84} then yields the following result.

\begin{prop}\label{prop:berghvariante}
	Let $E^1 = (E_0^1, E_1^1)$, $E^2 = (E_0^2,E_1^2)$ and $F = (F_0,F_1)$ be three couples of complex Banach spaces.
	Assume in addition that the (not necessarily linear) operator $N \colon \Sigma_{E^1} \times \Sigma_{E^2} \to \Sigma_F$ fulfills
	\begin{enumerate}[(i)]
		\item $f_k \in \sF(E_0^k,E_1^k)$ implies $N f = N(f_1,f_2) \in \sF(F_0,F_1)$, and
		\item $\| N(a_1,a_2)\|_{F_j} \le \mu(\| a_1 \|_{E_j^1},\| a_2 \|_{E_j^2})$, $a_1 \in E_j^1$, $a_2 \in E_j^2$, $j=0,1$,
	\end{enumerate}
	where $\mu$ is continuous and positive on $\R_{\ge 0} \times \R_{\ge 0}$, and it additionally satisfies the following monotonicity property
	$$
	\mu(x.y) \le \mu (\Tilde{x},y) \mbox{ and } \mu(x,y) \le \mu(x,\Tilde{y}) \mbox{ for all } x,y,\Tilde{x},\Tilde{y} \in \R_{\ge 0} \mbox{ such that } x \le \Tilde{x} \mbox{ and } y \le \Tilde{y}.
	$$
	Then, $a = (a_1,a_2) \in E_\theta^1 \times E_\theta^2$ implies that $N(a_1,a_2) \in F_\theta$, and it is valid that
	$$
	\| N(a_1,a_2) \|_{F_\theta} \le \mu(\| a_1 \|_{E_\theta^1}, \| a_2 \|_{E_\theta^2}).
	$$
\end{prop}

\section{Proof of the main result}
\label{sec:proof main result}

This section is dedicated to showing the main result. Let us first explain the general strategy. 

In order, to apply \autoref{prop:quasilinear} to \eqref{eq:transformed system}, we reformulate \eqref{eq:transformed system} as a non-autonomous quasilinear abstract Cauchy problem in \autoref{ssec:quasilinear aCP}.
Afterwards, we verify the conditions of \autoref{ass:quasilinear thm}.
We start by checking that the transformed operator has the property of maximal $\rL^p$-regularity in \autoref{ssec:A3}, for which we use decoupling techniques to account for the coupling condition on the interface.
Besides, employing the structure and regularity of the coordinate transform as explained in \autoref{rmk:procedure Z and Y from xi and Omega} and \autoref{prop:dep Z and Y on xi and Omega}, we establish Lipschitz estimates for the transformed operator in \autoref{ssec:A1} and for the transformed right-hand side in \autoref{ssec:A2}.
The latter aspect relies on an application of \autoref{prop:berghvariante}.
Finally, the local well-posedness of \eqref{eq:transformed system} follows from \autoref{prop:quasilinear}. 

In order to conclude \autoref{thm:main thm transformed}, recall from \autoref{rmk:procedure Z and Y from xi and Omega} that the matrix $Q$ and the transforms $Z$ and $Y$ can be constructed  from $(\xi,\Omega)$. 
As a consequence of \autoref{thm:main thm transformed}, we obtain our main result \autoref{thm:main theorem} by performing the backward change of variables and coordinates given in \autoref{sec:coordinate transform}.

\subsection{Reformulation of \eqref{eq:transformed system} as a non-autonomous quasilinear abstract Cauchy problem}
\label{ssec:quasilinear aCP}
\

We start by initializing the Banach spaces
\begin{equation}\label{eq:ground space}
    \tsX_0 = \tX_0 \times \R^3 = \rL^q(\cD;\R^2) \times \rL^q(\cD) \times \rL^q(\cD) \times \R^2 \times \R.
\end{equation}

For fixed $t \in \R_+$ and $\tw_0 = (\tu_0,\thseaice_0,\ta_0,\xi_0,\Omega_0)$, let $\tAHm(t,\tw_0)$ be given by
\begin{equation}\label{eq:maximal tranformed sea ice op}
    \tAHm(t,\tw_0) \tu := \frac{1}{\rice \thseaice_0} \tAH(t,\tw_0) \tu, \enspace D(\tAHm(t,\tw_0)) := \{ \tu \in \rW^{2,q}(\cD;\R^2) : \tu = 0 \mbox{ on } \partial \cO \},
\end{equation}
where the transformed Hibler operator $\tAH(t,\tw_0) \tu$ is defined in \eqref{eq:Hibler op transformed}.
Note that its domain is independent of $t$ and $\tw_0$. 
Further, the lower order terms lead to the operator
\begin{equation}\label{eq:transformed lower order terms}
	\tB(t,\tw_0) \binom{\thseaice}{\ta} := \frac{1}{\rice \thseaice_0}\tB_1(t,\tw_0)\thseaice + \frac{1}{\rice \thseaice_0}\tB_2(t,\tw_0) \ta,
\end{equation}
where $\tB_1(t,\tw_0)$ and $\tB_2(t,\tw_0)$ are given in \eqref{eq:B_1,2}. Finally, the transformed Laplacians $\tD(t,\tw_0)$ are given by
\begin{equation}\label{eq:transformed laplacians}
	\tD(t,\tw_0) \binom{\thseaice}{\ta} := \binom{d_h \tL(t,\xi_0,\Omega_0) \thseaice}{d_h \tL(t,\xi_0,\Omega_0) \ta}, \enspace
	D(\tD(t,\tw_0)) := \left\{ (\thseaice,\ta) \in \rW^{2,q}(\cD)^2 : \partial_\nu \thseaice = \partial_\nu \ta = 0 \mbox{ on } \partial \cD \right\} ,
\end{equation}
where $\tL(t,\xi_0,\Omega_0)$ is defined as in \eqref{eq:transformed Laplacian h,a}. 
Moreover, we denote by $\tr_\Gamma \colon \rW^{2,q}(\cD;\R^2) \to \rW^{2-\nicefrac{1}{q},q}(\Gamma;\R^2)$ the trace operator on $\Gamma$, and we define the coupling operator
\begin{equation}\label{eq:coupling operator}
	R \colon \R^3 \to \rW^{2-\nicefrac{1}{q}}(\Gamma,\R^2), \enspace \binom{\xi}{\Omega}
	\mapsto (\Omega y^{\perp} + \xi) \mathds{1}_{\Gamma},
\end{equation}
where $\mathds{1}_\Gamma$ denotes the constant $1$-function on $\Gamma$.

For $\tw_0 = (\tu_0,\thseaice_0,\ta_0,\xi_0,\Omega_0)$ we consider $\tAHm(t,\tw_0)$ as in \eqref{eq:maximal tranformed sea ice op}, $\tB(t,\tw_0)$ as in \eqref{eq:transformed lower order terms}, $\tD(t,\tw_0)$ as in \eqref{eq:transformed laplacians} and $R$ as in \eqref{eq:coupling operator}, and we introduce the family of operator matrices $\tsA(t,\tw_0)$ given by
\begin{equation}
	\tsA(t,\tw_0) 
	:= 
	\begin{pmatrix}
		\tAHm(t,\tw_0) & -\tB(t,\tw_0) & 0 \\
		0 & -\tD(t,\tw_0) & 0 \\
		0 & 0 & 0 
	\end{pmatrix}
\label{eq:tsA}
\end{equation}
with domain
\begin{equation}
	\tsX_1 := D(\tsA(t,\tw_0)) 
	:= \left\{
	\tw = (\tu,\thseaice,\ta,\xi,\Omega) \in D(\tAHm(t,\tw_0)) \times D(\tD(t,\xi_0,\Omega_0)) : \tr_\Gamma \tu = R \binom{\xi}{\Omega}
	\right\} .
	\label{eq:domain tsA}
\end{equation}
Again, we observe that the domain is independent of $t$ and $\tw_0$.
We will see in \autoref{ssec:A3} that $\tsA(0,\tw_0)$ is a closed operator and hence $\tsX_1 = D(\tsA(t,\tw_0))$ is a Banach space with the graph norm. 
Furthermore, we denote by $\tsX_\gamma = (\tsX_0,\tsX_1)_{1-\nicefrac{1}{p},p}$ the trace space, which we will determine in terms of Besov spaces in \autoref{lem:trace space}.  

Next, we summarize the transformed nonlinear terms coming from the sea ice equations in \eqref{eq:transformed system} in
\begin{equation}\label{eq:G1 transformed}
	\tsG_1(t,\tw)
	= \begin{pmatrix}
		\frac{1}{\rice \thseaice} f_1(\tu,\thseaice) - \sum_{i,j=1}^2 \tu_i (\partial_i Y_j) \partial_j \tu\\
		S_h(\thseaice,\ta) - \sum_{i,j=1}^2 (\partial_i Y_j) \bigl(\tu_i \partial_j \thseaice + \thseaice \partial_j \tu_i\bigr)\\
		S_a(\thseaice,\ta) - \sum_{i,j=1}^2 (\partial_i Y_j) \bigl(\tu_i \partial_j \ta + \ta \partial_j \tu_i\bigr)
	\end{pmatrix},
\end{equation}
with $f_1$ is as in \eqref{eq:rhsseaice}, $S_h$ as in \eqref{eq:S_h} and $S_a$ as in \eqref{eq:S_a}, where \autoref{rmk:procedure Z and Y from xi and Omega} guarantees that the right-hand side is a function depending only on $(t,\tw)$.
Similarly, we introduce $\tsG_2$ for the transformed nonlinear terms in the equations for the motion of the rigid body, i.e., 
\begin{equation}\label{eq:G2 transformed}
	\tsG_2(t,\tw)
	= \begin{pmatrix}
		\frac{1}{\mbody} \tF - \frac{1}{\mbody} \int_{\Gamma} \sT_\delta(\tu,\thseaice,\ta) \tn \,\mathrm{d} S\\
		I^{-1} \tN - I^{-1} \int_{\Gamma} y^\perp \sT_\delta(\tu,\thseaice,\ta) \tn \,\mathrm{d} S
	\end{pmatrix}
\end{equation}
for $\tw = (\tu,\thseaice,\ta,\xi,\Omega)$. 
We then set $\tsG(t,\tw) := \binom{\tsG_1(t,\tw)}{\tsG_2(t,\tw)}$.
Taking into account \autoref{rmk:procedure Z and Y from xi and Omega}, we introduce the notation
\begin{equation}\label{eq:sM}
	\tsM_1(t,\xi,\Omega) \tv = \begin{pmatrix}
		\sum_{j=1}^2 \dot{Y_j} \partial_j \tu\\ \sum_{j=1}^2 \dot{Y_j} \partial_j \thseaice\\ \sum_{j=1}^2 \dot{Y_j} \partial_j \ta   
	\end{pmatrix}
\end{equation}
for $\tv= (\tu,\thseaice,\ta)$ and $\tsM(t,\xi,\Omega) := \binom{\tsM_1(t,\xi,\Omega)}{0}$.
Finally, we define
\begin{equation}
	\tsF(t,\tw) := \tsG(t,\tw) - \tsM(t,\xi,\Omega) \tw + (0,0,0,\Omega \xi^\perp,0)^\T.
	\label{eq:rhs operator} 
\end{equation}
Therefore, \eqref{eq:transformed system} can be written as a non-autonomous quasilinear abstract Cauchy problem
\begin{equation}
	\left\{
	\begin{aligned}
		\tw_t + \tsA(t,\tw) \tw &= \tsF(t,\tw), \quad t \in [0,T], \\
		\tw(0) &= \tw_0,
	\end{aligned}
	\right.
	\label{eq:ACP}
\end{equation}
with initial data $\tw_0 = (\tu_0,\thseaice_0,\ta_0,\xi_0,\Omega_0) \in V$, where $V$ is as made precise in \eqref{eq:V}.
In the next subsections, we verify that \eqref{eq:ACP} is in the framework of \autoref{subsec:quasilin result}, i.e., we check that $\tsA$ as in \eqref{eq:tsA} and $\tsF$ as in \eqref{eq:rhs operator} satisfy the conditions (A1), (A2) and (A3) from \autoref{ass:quasilinear thm}.

\subsection{Maximal regularity of the linearized operator matrix}\label{ssec:A3}

\

In this subsection, we give an abstract tool to verify the assumption that for $w_0 \in V$ fixed, $\tsA(0,w_0)$ admits maximal $\rL^p$-regularity on the ground space $\tsX_0$ as made precise in \eqref{eq:ground space}.
Throughout this subsection, we employ the notation $b = (h,a)$, $z = (\eta,\omega)$.

Recalling $\AH(w_0)$ from \eqref{eq:Hibler op} as well as $\tAHm(0,w_0)$ and $D(\AHm) = D(\tAHm(0,w_0))$ from \eqref{eq:maximal tranformed sea ice op}, we introduce
\begin{equation*}\label{eq:operators for fixed time}
    \left\{
    \begin{aligned}
        \AHm(w_0) u &:= \tAHm(0,w_0)u = \frac{1}{\rice h_0}\AH(w_0)u, \mbox{ for } u \in D(\AHm) := D(\tAHm(0,w_0)), \\
        \rD b &:= \diag(d_h \Delta, d_a \Delta)b, \mbox{ for } b \in D(\rD) := \{b \in \rW^{2,q}(\cD)^2 : \partial_\nu h = \partial_\nu a = 0 \mbox{ on } \partial \cD\},\\
        \rB(w_0)b &:= \frac{\partial_h P(h_0,a_0)}{2 \rice h_0} \nabla h +  \frac{\partial_a P(h_0,a_0)}{2 \rice h_0} \nabla a, \mbox{ for } b \in D(\rD).
    \end{aligned}
    \right.
\end{equation*}
Here we used that the transformation $Y$ satisfies $Y(0) = \Id$ by definition, and therefore, the transformed variables and the usual variables coincide. 
We then get
\begin{equation*}
    \sA := \tsA(0,w_0) = 
	\begin{pmatrix}
		\tAHm(0,w_0) & -\tB(0,w_0) & 0 \\
		0 & -\tD(0,w_0) & 0 \\
		0 & 0 & 0 
	\end{pmatrix} = 	\begin{pmatrix}
		\AHm(w_0) & -\rB(w_0) & 0 \\
		0 & -\rD & 0 \\
		0 & 0 & 0 
	\end{pmatrix}
\end{equation*}
with coupled domain $\tsX_1$ as in \eqref{eq:domain tsA}.
To simplify the notation, we omit the entry $w_0$, since it is fixed. 
For the trace $\tr_\Gamma$ and the coupling operator $R$ as in \eqref{eq:coupling operator}, the coupled domain then rewrites as
\begin{equation*}
    \tsX_1 = D(\sA) = \{(u,b,z) \in D(\AHm) \times D(\rD) \times \R^3 : \tr_\Gamma u = Rz\}.
\end{equation*}
Next, we introduce the \emph{Hibler operator with homogeneous Dirichlet boundary conditions} $\AHD$ on $\rL^q(\cD;\R^2)$ as investigated in \cite[Sections~4 and 6]{BDHH:22}.
In the present framework, it is given by
\begin{equation*}
	\AHD u := \AHm u, \qquad D(\AHD) := \ker(\tr_\Gamma) = \rW^{2,q}(\cD;\R^2) \cap \rW_0^{1,q}(\cD;\R^2).
\end{equation*} 
Now, the \emph{decoupled operator matrix}
$\sA_0 \colon D(\sA_0) \subset \tsX_0 \to \tsX_0$ is given by
\begin{equation}\label{eq:decoupled operator matrix}
	\sA_0 = 
	\begin{pmatrix}
		\AHD & -\rB & 0 \\
		0 & -\rD & 0 \\
		0 & 0 & 0
	\end{pmatrix},
	\qquad D(\sA_0) = D(\AHD) \times D(\rD) \times \R^3,
\end{equation}
and we equip $D(\sA_0)$ with the graph norm.
By \cite[Theorem~4.4 and Lemma~6.1]{BDHH:22}, there is $\lambda_0 \in \R$ such that for all $\lambda > \lambda_0$ it holds that $\AHD + \lambda$ has the property of maximal $\rL^p$-regularity, and it follows that $\AHD + \lambda$ then is also invertible.
Therefore, for such $\lambda$, we introduce the translated versions of $\sA$ and $\sA_0$ as $\sA_\lambda := \sA + \diag(\lambda,0,0)$ and $\sA_{0,\lambda} := \sA_0 + \diag(\lambda,0,0)$
which does not affect the domains of the operators, i.e., $D(\sA_\lambda) = D(\sA) = \tsX_1$ and $D(\sA_{0,\lambda}) = D(\sA_0)$.

The aim now is to show that $\sA_\lambda$ admits maximal $\rL^p$-regularity on $\tsX_0$ by invoking the maximal regularity of $\AHD + \lambda$ and exploiting the upper triangular structure of the operator matrix.
To this end, we use a method of decoupling, and we first argue that 
\begin{equation}\label{eq:def L0}
    L_0 = \left(\tr_\Gamma\big|_{\ker (\AHm + \lambda)}\right)^{-1}
\end{equation}
is well-defined and continuous:

\begin{lem}\label{lem:L_0 existence and continuity}
	The operator $\tr_\Gamma\big|_{\ker (\AHm + \lambda)} \in \sL(\rW^{2,q}(\cD;\R^2),\rW^{2-\nicefrac{1}{q},q}(\Gamma;\R^2))$ is continuously invertible, and $L_0$ defined in \eqref{eq:def L0} is thus bounded from $\rW^{2-\nicefrac{1}{q},q}(\Gamma;\R^2)$ to $\rW^{2,q}(\cD;\R^2)$.
\end{lem}
\begin{proof} 
We make use of  \cite[Theorem~4.7.1]{Tri:78} to deduce that $\tr_\Gamma$ is a retraction from $\rW^{2,q}(\cD;\R^2)$ onto $\rW^{2-\nicefrac{1}{q},q}(\Gamma;\R^2)$, so it is in particular surjective and continuous. 
The closed graph theorem, or, equivalently, the bounded inverse theorem, yields that the graph norm of $\AHm$ is equivalent to the $\rW^{2,q}$-norm, and closedness of $\AHm $ follows by the fact that it is an elliptic differential operator of second order.
In particular, we have argued that ${\tr_\Gamma \in \sL(\rW^{2,q}(\cD;\R^2),\rW^{2 - \nicefrac{1}{q}}(\Gamma;\R^2))}$.
The latter observation and closedness of $\AHm$ imply that
\begin{equation*}
    \binom{\AHm}{\tr_\Gamma} \colon D(\AHm) \to \rL^q(\cD;\R^2) \times \rW^{2-\nicefrac{1}{q},q}(\Gamma;\R^2)
\end{equation*}
is closed.
The invertibility of $\AHD + \lambda$ yields the existence and continuity of the above $L_0$ in view of \cite[Lemma 2.2]{CENN:03}.
Note that $\tr_\Gamma$ plays the role of $L$ in the abstract framework of \cite[Section~2]{CENN:03}.
\end{proof} 

The shape of $R$ reveals that it is especially bounded and $\im(R) \subset \rW^{2-\nicefrac{1}{q},q}(\Gamma;\R^2)$, so it follows that $L_0 R$ is bounded as the product of two bounded operators.
Consequently, the operator
\begin{equation}\label{eq:decoupling matrix S}
	\sS = \begin{pmatrix}
		\Id & 0 & - L_0 R\\
		0 & \Id & 0\\
		0 & 0 & \Id
	\end{pmatrix} \qquad \mbox{ is bounded with inverse } \qquad \sS^{-1} = \begin{pmatrix}
		\Id & 0 & L_0 R\\
		0 & \Id & 0\\
		0 & 0 & \Id
	\end{pmatrix}.    
\end{equation}
The following result establishes maximal regularity in the coupled setting.

\begin{prop}\label{prop: max. Reg.}
	Let $p,q \in (1,\infty)$ be such that \eqref{eq:condition p and q} holds true.
	Then for $w_0 \in V$, given in \eqref{eq:V}, and $\lambda > \lambda_0$, the operator matrix $\sA_\lambda$ admits maximal $\rL^p$-regularity on $\tsX_0$.
\end{prop}

\begin{proof}
Using $\im(L_0) \subset D(\AHm)$, we derive that
\begin{equation*}
    \begin{aligned}
        \sS D(\sA)
	    &= \left \lbrace \begin{pmatrix}
		u - L_0 R z \\ b \\ z
	    \end{pmatrix} : u \in D(\AHm), \enspace b \in D(\rD), \enspace z \in \R^3, \enspace \tr_\Gamma u = R z \right \rbrace\\
	    &= \left \lbrace \begin{pmatrix}
		\hat{u} \\ b \\ z
	    \end{pmatrix} : \hat{u} \in D(\AHm), \enspace b \in D(\rD), \enspace z \in \R^3, \enspace \tr_\Gamma \hat{u} = 0 \right \rbrace\\
	    &= D(\AHD) \times D(\rD) \times \R^3 = D(\sA_{0,\lambda}).
    \end{aligned}
\end{equation*}
On the other hand, employing that $\im(L_0) \subset \ker(\AHm + \lambda)$, we deduce that
\begin{equation*}
	\sS \sA_\lambda \sS^{-1} = \begin{pmatrix}
		\AHm + \lambda & -\rB & (\AHm + \lambda) L_0 R\\
		0 & -\rD & 0\\
		0 & 0 & 0
	\end{pmatrix} = 
	\begin{pmatrix}
		\AHm + \lambda & -\rB & 0\\
		0 & -\rD & 0\\
		0 & 0 & 0
	\end{pmatrix} = \sA_{0,\lambda},   
\end{equation*}
where we used that $\sS D(\sA)$ is precisely the domain of $\sA_{0,\lambda}$.
The upper triangular structure of $\sA_{0,\lambda}$, the well known fact that the Neumann Laplacian operator on $\rL^q(\cD)$ has the property of maximal $\rL^p$-regularity and the aforementioned result from \cite{BDHH:22} imply that $\sA_{0,\lambda}$ admits maximal $\rL^p$-regularity.
The claim follows now from $\sA_\lambda = \sS^{-1} \sA_{0,\lambda} \sS$. 
\end{proof}

We next verify that the trace space in the coupled setting can be represented in terms of the matrix $\sS^{-1}$ and the trace space $\tsY_\gamma$ in the decoupled setting, and we characterize the coupled trace space in terms of Besov spaces.
The ground space in the decoupled setting is also given by $\tsX_0$, while $D(\sA_0)$ as in \eqref{eq:decoupled operator matrix} represents the regularity space.

\begin{lem}\label{lem:trace space}
	For $\sS$ and $\sS^{-1}$ as in \eqref{eq:decoupling matrix S}, $\tsX_0$ as in \eqref{eq:ground space}, $\tsX_1$ as in \eqref{eq:domain tsA} and $D(\sA_0)$ as in \eqref{eq:decoupled operator matrix}, $\tsX_\gamma = (\tsX_0,\tsX_1)_{1-\nicefrac{1}{p},p}$ and $\tsY_\gamma = (\tsX_0,D(\sA_0))_{1-\nicefrac{1}{p},p}$, it holds that
	\begin{equation*}
	    \tsX_\gamma = \sS^{-1}(\tsY_\gamma) = \left\{(u,b,z) \in \rB_{qp}^{2 - \nicefrac{2}{p}}(\cD;\R^2) \times \rB_{qp}^{2 - \nicefrac{2}{p}}(\cD)^2 \times \R^3 : u = \eta + \omega (x - x_c)^\perp\right\}.
	\end{equation*}
\end{lem}

\begin{proof}
    It is well known that $(\rL^q(\cD),\rW^{2,q}(\cD))_{1-\nicefrac{1}{p},p} = \rB_{qp}^{2 - \nicefrac{2}{p}}(\cD)$, see e.g.\ \cite{Tri:78} or \cite[Chapter~5]{Ama:93}.
    
	We recall from the above arguments that $\sA_0$ and $\sA$ generate analytic semigroups $T_{\sA_0}$ and $T_{\sA}$ on $\tsX_0$.
	Moreover, $\sA = \sS^{-1} \sA_0 \sS$, and it also follows that $T_{\sA} = \sS^{-1} T_{\sA_0} \sS$.
	By definition of the trace space, an insertion of the relation of the operators and semigroups, and using $x = \sS^{-1} y$ for some $y \in \tsX_0$ in conjunction with easy functional analytic arguments involving the boundedness of $\sS$, we have
	\begin{equation*}
	    \begin{aligned}
		\tsX_\gamma 
		&= \left\{x \in \tsX_0 : [x]_{1-\nicefrac{1}{p},p} := \left(\int_0^\infty \| t^{\nicefrac{1}{p}} \sA T_{\sA} x \|_{\tsX_0} ^p \,\mathrm{d}t/t \right)^{\nicefrac{1}{p}} < \infty\right\}\\
		&= \left\{\sS^{-1} y \in \tsX_0 : [\sS^{-1} y]_{1-\nicefrac{1}{p},p} := \left(\int_0^\infty \| t^{\nicefrac{1}{p}} \sS^{-1} \sA_0 T_{\sA_0} y \|_{\tsX_0} ^p \,\mathrm{d}t/t \right)^{\nicefrac{1}{p}} < \infty\right\}\\
		&= \sS^{-1}(\tsY_\gamma),	        
	    \end{aligned}
	\end{equation*}
	so the proof is completed by recalling the shape of $\sS^{-1}$.
\end{proof}

\subsection{Lipschitz estimates of the operator matrix}\label{ssec:A1}

\

We start verifying the aspect (A1) of \autoref{ass:quasilinear thm}.

\begin{lem}\label{lem:continuity in A1}
    Let $p,q \in (1,\infty)$ satisfy \eqref{eq:condition p and q}, consider $V$ as in \eqref{eq:V} and recall $\tsA$ from \eqref{eq:tsA}, $\tsX_0$ from \eqref{eq:ground space} as well as $\tsX_1$ from \eqref{eq:domain tsA}.
    Then the map
    \begin{equation*}
        [0,T] \times V \to \sL(\tsX_1,\tsX_0), \enspace (t,\tw_0) \mapsto \tsA(t,\tw_0)
    \end{equation*}
    is continuous.
\end{lem}

\begin{proof}
In order to show the continuity of the operator matrix $[0,T] \times V \to \sL(\tsX_1,\tsX_0), \enspace (t,\tw_0) \mapsto \tsA(t,\tw_0)$, it suffices to show continuity of $(t,\tw_0) \mapsto \tAHm(t,\tw_0)$, $(t,\tw_0) \mapsto -\tB(t,\tw_0)$ and $(t,\tw_0) \mapsto -\tD(t,\tw_0)$ separately. 

First, we observe that for $\tw_0 = (\tu_0,\thseaice_0,\ta_0,\xi_0,\Omega_0) \in V$, the angular and translational velocities $\xi_0$ and $\Omega_0$ are independent of time, so it is in particular valid that $(\xi_0,\Omega_0) \in \rW^{1,p}(0,T;\R^2) \times \rW^{1,p}(0,T)$.
Consequently, we are in the framework of \autoref{rmk:procedure Z and Y from xi and Omega} as well as \autoref{prop:dep Z and Y on xi and Omega}.
As a result, we obtain continuity of
\begin{equation}\label{eq:continuity}
    \begin{aligned}
        &[0,T] \times \R^3 \to \rL^\infty(\R^2) \colon (t,\xi_0,\Omega_0) \mapsto \partial_{i} Y \quad \text{as well as} \\
	    &[0,T] \times \R^3 \to \rL^\infty(\R^2) \colon (t,\xi_0,\Omega_0) \mapsto \partial_i \partial_j Y, \quad \text{and then also of}\\
        &[0,T] \times \R^3 \to \rL^\infty(\R^2) \colon  (t,\xi_0,\Omega_0) \mapsto g^{ij},
    \end{aligned}
\end{equation}
where $g^{ij}$ denotes the contravariant tensor as introduced in \eqref{eq:gij}.

Using H\"older's inequality in conjunction with the fact that $\tL(t,\xi_0,\Omega_0)$ from \eqref{eq:transformed Laplacian h,a} is a differential operator of second order, we deduce from \eqref{eq:continuity} that $[0,T] \times \R^3 \to \sL(\rW^{2,q}(\cD),\rL^q(\cD)) \colon (t,\xi_0,\Omega_0) \mapsto \tL(t,\xi_0,\Omega_0)$ is continuous.
It readily follows from the shape of $-\tD(t,\xi_0,\Omega_0)$ from \eqref{eq:transformed laplacians} that
\begin{equation*}
    [0,T] \times V \to \sL(\rW^{2,q}(\cD)^2,\rL^q(\cD)^2) \colon (t,\tw_0) \mapsto - \tD(t,\tw_0)
\end{equation*}
is continuous.

Next, we note that for $\tw_0 = (\tu_0,\thseaice_0,\ta_0,\xi_0,\Omega_0) \in V$, we get $\frac{1}{\rice \thseaice_0} < \frac{1}{\rice \kappa} < \infty$.
In addition, we infer continuity of
\begin{equation}\label{eq:cont of factors}
    \rL^\infty(\cD)^2 \to \rL^\infty(\cD), (\thseaice_0,\ta_0) \mapsto \mathrm{e}^{-c(1-\ta)} \quad \text{and} \quad \rL^\infty(\cD)^2 \to \rL^\infty(\cD), (\thseaice_0,\ta_0) \mapsto \thseaice_0 \mathrm{e}^{-c(1-\ta)},
\end{equation}
so the embedding
\begin{equation}\label{eq:V emb Linfty}
    V \hra \rL^\infty(\cD)^4 \times \R^3,
\end{equation}
following from \eqref{eq:emb Besov C1} and \autoref{lem:trace space}, and H\"older's inequality yield continuity of
\begin{equation*}
    [0,T] \times V \to \sL(\rW^{1,q}(\cD)^2,\rL^q(\cD)^2) \colon (t,\tw_0) \mapsto - \tB(t,\tw_0)
\end{equation*}
upon recalling the shape of $\tB(t,\tw_0)$ from \eqref{eq:transformed lower order terms}. 

Finally, we show the continuity of $\tAHm$:
In view of \eqref{eq:maximal tranformed sea ice op}, the above observation that $\frac{1}{\rice \thseaice_0} < \frac{1}{\rice \kappa} < \infty$ for $\tw_0 = (\tu_0,\thseaice_0,\ta_0,\xi_0,\Omega_0) \in V$ implies that it suffices to show the continuity of $\tAH$ defined in \eqref{eq:Hibler op transformed}.
By H\"older's inequality, \eqref{eq:continuity}, \eqref{eq:cont of factors}, \eqref{eq:V emb Linfty}, the shapes of $a_{ij}^{kl}$ as in \eqref{eq:coeff} and $a_{ij}^{klm}$ as in \autoref{sec:coordinate transform} and the continuous dependence of the coefficients $a_{ij}^{kl}$ on $\tu_0$, $\thseaice_0$ as well as $\ta_0$, see \cite[Section~6]{BDHH:22}, we conclude continuity of 
\begin{equation*}
    [0,T] \times V \to \sL(\rW^{2,q}(\cD)^2,\rL^q(\cD)^2) \colon (t,\tw_0) \mapsto \tAH(t,\tw_0).
\end{equation*}

The assertion of the lemma then follows by concatenating the previous arguments and observing that
\begin{equation}\label{eq:emb tsX_1}
    \tsX_1 \hra \rW^{2,q}(\cD)^4 \times \R^3
\end{equation}
by virtue of $\tsX_1 = \sS^{-1} (D(\sA_0))$, see the proof of \autoref{prop: max. Reg.} for this relation of $D(\sA_0)$ as in \eqref{eq:decoupled operator matrix}, and the classical embedding $D(\sA_0) \hra \rW^{2,q}(\cD)^4 \times \R^3$.
\end{proof}

Now we show the Lipschitz continuity of the system matrix. 

\begin{lem}\label{lem:lipschitz systemmatrix}
	Let $p,q \in (1,\infty)$ be such that \eqref{eq:condition p and q} is valid, and consider $\tw_0 \in V$ fixed.
	Then there exists $R_0 > 0$ and a constant $L>0$ independent of $\tau$ such that $\overline{B}_{\tsX_\gamma}(\tw_0,R_0) \subset V$ and
	\begin{equation*}
		\| \tsA(\tau,\tw_1)\tu - \tsA(\tau,\tw_2)\tu \|_{\tsX_0} \le L \cdot \| \tw_1 - \tw_2 \|_{\tsX_\gamma} \cdot \| \tu \|_{\tsX_1}  
	\end{equation*}
	holds for all $\tw \in \tsX_1$, for all $\tw_1, \tw_2 \in \tsX_\gamma$ with $\| \tw_i - \tw_0 \|_{\tsX_\gamma} \le R_0$, $i=1,2$, and for $\tau \in [0,T]$, with $T>0$ sufficiently small.
\end{lem}
\begin{proof}
We recall that $\tw_i = (\tu_i,\thseaice_i,\ta_i,\xi_i,\Omega_i)$ for $i \in \{0,1,2\}$.
First, it follows from $\tw_1, \tw_2 \in V$ that $\xi_i$ and $\Omega_i$, $i=1,2$, are independent of time, so it holds that $(\xi_i,\Omega_i) \in \rW^{1,p}(0,T;\R^2) \times \rW^{1,p}(0,T)$, and we are thus in the setting of \autoref{rmk:procedure Z and Y from xi and Omega} and \autoref{prop:dep Z and Y on xi and Omega}.
Using that
\begin{equation*}
    K_i = \| \xi_i \|_{\rL^\infty(0,T;\R^2)} + \| \Omega_i \|_{\rL^\infty(0,T)} = | \xi_i | + | \Omega_i | \le c(R_0 + \| \tw_0 \|_{\tsX_\gamma}), \quad i=1,2, 
\end{equation*}
where we additionally exploited that $(\xi_i,\Omega_i)$ are independent of time.
Employing this observation in conjunction with \autoref{prop:dep Z and Y on xi and Omega}, we argue that
\begin{equation}\label{eq:lipschitz est Z and Y}
    \begin{aligned}
            \| \partial^\alpha Z_i \|_{\infty,\infty} + \| \partial^\alpha Y_i \|_{\infty,\infty} &\le \Tilde{c} (R_0 + \| \tw_0 \|_{\tsX_\gamma}), \quad i=1,2, \quad \text{and}\\
            \| \partial^\beta (Z_1 - Z_2) \|_{\infty,\infty} + \| \partial^\beta (Y_1 - Y_2) \|_{\infty,\infty} &\le \Tilde{c}(R_0 + \| \tw_0 \|_{\tsX_\gamma}) T (|\xi_1 - \xi_2 | + | \Omega_1 - \Omega_2 |)
        \end{aligned}
\end{equation}
for all multi-indices $\alpha, \beta$ such that $1 \le |\alpha| \le 3$ and $0 \le |\beta| \le 3$.
As above, the index $i$ indicates that $Z_i$ and $Y_i$ correspond to $(\xi_i,\Omega_i)$, $i=1,2$.

As in the proof of \autoref{lem:continuity in A1}, it is sufficient to verify the Lipschitz properties of the entries $\tAHm(t,\tw)$, $-\tB(t,\tw)$ and $-\tD(t,\tw)$ separately.
First, using H\"older's inequality, for $\tw_1, \tw_2 \in \overline{B}_{\tsX_\gamma}(\tw_0,R_0) \subset V$, $\tw \in \tsX_1$ and $\tau \in [0,T]$, we get for $\tL$ as in \eqref{eq:transformed Laplacian h,a} that
\begin{equation*}
    \begin{aligned}
        \|\tL(\tau,\xi_1,\Omega_1)\thseaice-
		\tL(\tau,\xi_2,\Omega_2)\thseaice\|_{\rL^q(\cD)} &\le \sum_{j = 1}^2
		\| \Delta Y_j (\tau,\xi_1,\Omega_1)
		- \Delta Y_j (\tau,\xi_2,\Omega_2) \|_{\rL^\infty(\cD)} 
		\cdot \| \partial_j \thseaice \|_{\rL^q(\cD)}\\
        &\quad + \sum_{j,k=1}^2 \| g^{ij} (\tau,\xi_1,\Omega_1)
		- g^{ij} (\tau,\xi_2,\Omega_2) \|_{\rL^\infty(\cD)} 
		\cdot \| \partial_k \partial_j \thseaice \|_{\rL^q(\cD)}\\
        &\le C (R_0 + \| \tw_0 \|_{\tsX_\gamma}) T \| \tw_1 - \tw_2 \|_{\tsX_\gamma} \| \tw \|_{\tsX_1}
    \end{aligned}
\end{equation*}
by virtue of \eqref{eq:lipschitz est Z and Y} and the embedding from \eqref{eq:emb tsX_1}.
We emphasize that the Lipschitz constant is independent of $\tau \in [0,T]$.
Furthermore, we find a similar estimate for $\|\tL(\tau,\xi_1,\Omega_1)\ta- \tL(\tau,\xi_2,\Omega_2)\ta\|_{\rL^q(\cD)}$, so for $\tD(t,\xi_i,\Omega_i)$ as in \eqref{eq:transformed laplacians}, we get
\begin{equation*}
    \left\|\tD(\tau,\xi_1,\Omega_1)\binom{\thseaice}{\ta}-
		\tD(\tau,\xi_2,\Omega_2)\binom{\thseaice}{\ta}\right\|_{\rL^q(\cD)^2} \le C (R_0 + \| \tw_0 \|_{\tsX_\gamma}) T \| \tw_1 - \tw_2 \|_{\tsX_\gamma} \| \tw \|_{\tsX_1}.
\end{equation*}

Next, making use of H\"older's inequality, \eqref{eq:lipschitz est Z and Y}, $\tw_i \in V$, with $V$ as in \eqref{eq:V}, resulting in $\mathrm{e}^{-c(1-\ta_i)}$ being finite, the embedding \eqref{eq:V emb Linfty} as well as the Lipschitz estimates of terms associated to sea ice, see \cite[Lemma~6.2]{BDHH:22}, for $\tB_1(t,\tw_i)$ as in \eqref{eq:B_1,2}, we infer that
\begin{equation*}
    \begin{aligned}
        \| \tB_1(\tau,\tw_1) \thseaice - \tB_1(\tau,\tw_2) \thseaice \|_{\rL^q(\cD)^2} &\le C \sum_{j=1}^2 \| \mathrm{e}^{-c(1-\ta_1)} \partial_i Y_{1,j} - \mathrm{e}^{-c(1-\ta_2)} \partial_i Y_{2,j} \|_{\rL^\infty(\cD)^2} \| \thseaice \|_{\rW^{1,q}(\cD)}\\
        &\le C\left(\| \mathrm{e}^{-c(1-\ta_1)} (\partial_i Y_1 - \partial_i Y_2) \|_{\rL^\infty(\cD)^2}\right.\\
        &\quad \left. + \| (\mathrm{e}^{-c(1-\ta_1)} - \mathrm{e}^{-c(1-\ta_2)}) \partial_i Y_2 \|_{\rL^\infty(\cD)^2}\right) \| \thseaice \|_{\rW^{1,q}(\cD)}\\
        &\le C (R_0 + \| \tw_0 \|_{\tsX_\gamma}) (T + 1) \| \tw_1 - \tw_2 \|_{\tsX_\gamma} \| \tw \|_{\tsX_1}.
    \end{aligned}
\end{equation*}
Similar arguments lead to an analogous estimate for $\tB_2(\tau,\tw_1) \thseaice - \tB_2(\tau,\tw_2) \thseaice$.
Additionally invoking that $\frac{1}{\rice \thseaice_i} < \frac{1}{\rice \kappa} < \infty$ by $\tw_i \in V$ and employing the embedding from \eqref{eq:emb tsX_1}, we deduce that for $\tB(t,\tw_i)$ as in \eqref{eq:transformed lower order terms}, it holds that
\begin{equation*}
    \| \tB(\tau,\tw_1) - \tB(\tau,\tw_2) \|_{\rL^\infty(\cD)^2} \le C (R_0 + \| \tw_0 \|_{\tsX_\gamma}) (T + 1) \| \tw_1 - \tw_2 \|_{\tsX_\gamma} \| \tw \|_{\tsX_1}.
\end{equation*}

Recalling the shapes of $\tAHm$ and $\tAH$ from \eqref{eq:maximal tranformed sea ice op} and \eqref{eq:Hibler op transformed}, respectively, and exploiting H\"older's inequality, \eqref{eq:lipschitz est Z and Y}, the embeddings \eqref{eq:V emb Linfty} and \eqref{eq:emb tsX_1}, $\tw_i \in V$ as well as the estimates of the quasilinear terms in \cite[Lemma~6.2]{BDHH:22}, we find by means of a similar procedure as in the estimates above that
\begin{equation*}
    \| \tAHm(\tau,\tw_1) \tw - \tAHm(\tau,\tw_2) \tw \|_{\rL^q(\cD)^2} \le C (R_0 + \| \tw_0 \|_{\tsX_\gamma}) (T + 1) \| \tw_1 - \tw_2 \|_{\tsX_\gamma} \| \tw \|_{\tsX_1}.
\end{equation*}

Concatenating the above estimates and observing that the associated Lipschitz constants do only depend on $R_0$, $\|\tw_0\|_{\tsX_\gamma}$ and $T$, but not on $\tau \in [0,T]$, we conclude the statement of the lemma.
\end{proof}

\subsection{Lipschitz properties of the right-hand side}\label{ssec:A2}

\

This subsection is devoted to checking the aspect (A2) of \autoref{ass:quasilinear thm}.

\begin{lem}\label{lem:(i)-(iii) in (A2)}
    Let $p,q \in (1,\infty)$ be such that \eqref{eq:condition p and q} holds true, and let $V$ be as in \eqref{eq:V}, $\tsF$ as in \eqref{eq:rhs operator}, $\tsX_0$ as in \eqref{eq:ground space} and $\tsX_1$ as in \eqref{eq:domain tsA}, and suppose that $F \in \rL^p(0,T;\R^2)$ as well as $N \in \rL^p(0,T)$.
    Then
    \begin{enumerate}[(i)]
        \item $\tsF(\cdot,\tw_0)$ is measurable for every $\tw_0 \in V$,
        \item $\tsF(\tau,\cdot) \in \rC(V,\tsX_0)$ for almost all $\tau \in [0,T]$, and
        \item $\tsF(\cdot,\tw_0) \in \rL^p(0,T;\tsX_0)$ is valid for every $\tw_0 \in V$.
    \end{enumerate}
\end{lem}

\begin{proof}
The assumption $\tw_0 \in V \subset \tsX_\gamma$ ensures that the terms involved are well-defined.
In addition, for $\tw_0 = (\tu_0,\thseaice_0,\ta_0,\xi_0,\Omega_0)) \in V$, we deduce that $\xi_0$ and $\Omega_0$ are independent of time, so it is especially valid that $(\xi_0,\Omega_0) \in \rW^{1,p}(0,T;\R^2) \times \rW^{1,p}(0,T)$.
Consequently, it is justified to use \autoref{rmk:procedure Z and Y from xi and Omega} as well as \autoref{prop:dep Z and Y on xi and Omega}.

The shape of $\tsF$ as seen in \eqref{eq:rhs operator}, see also \eqref{eq:G1 transformed}, \eqref{eq:G2 transformed} as well as \eqref{eq:sM}, in conjunction with the regularity of the $Z$ and $Y$ in time and space as discussed in \autoref{prop:dep Z and Y on xi and Omega} and the measurability of the initial objects on the moving sea ice domain yield that (i) is satisfied.

Employing similar arguments as in the proof of \autoref{lem:continuity in A1}, we find that (ii) is also valid.
Property~(iii) follows by similar arguments as (i), and in particular, we take the assumptions concerning $F$ and $N$ from \autoref{thm:main theorem} into account and observe the way they are transformed in \autoref{sec:coordinate transform} as well as $Q \in \rW^{2,p}(0,T;\R^{2 \times 2}) \hra \rL^\infty(0,T;\R^{2 \times 2})$, see \autoref{rmk:procedure Z and Y from xi and Omega}.
\end{proof}

It remains to verify the Lipschitz property of $\tsF$, i.e., aspect (A2)(iv) in \autoref{ass:quasilinear thm}.
Using a similar strategy as for the treatment of $\tsA$ in \autoref{ssec:A1}, i.e., making use of \autoref{rmk:procedure Z and Y from xi and Omega} as well as \autoref{prop:dep Z and Y on xi and Omega} and invoking the respective estimates of the sea ice terms from \cite[Lemma~6.2]{BDHH:22}, we get the following result for $\tsG_1$ as in \eqref{eq:G1 transformed}.

\begin{lem}\label{lem:treatment of G1}
    Let $p,q \in (1,\infty)$ be such that \eqref{eq:condition p and q} is satisfied, and let $\tw_0 \in V$ be fixed.
    Then there is $R_0 > 0$ and a constant $L > 0$ such that $\overline{B}_{\tsX_\gamma}(\tw_0,R_0) \subset V$, and
    \begin{equation*}
        \| \tsG_1(\tau,\tw_1) - \tsG_1(\tau,\tw_2) \|_{\tX_0} \le L \cdot \| \tw_1 - \tw_2 \|_{\tsX_\gamma}
    \end{equation*}
    holds for all $\tw_1, \tw_2 \in \tsX_\gamma$ with $\| \tw_i - \tw_0 \|_{\tsX_\gamma} \le R_0$, $i=1,2$, and for $\tau \in [0,T]$, where $T>0$ is sufficiently small.
\end{lem}

The treatment of $\tsG_2$ as in \eqref{eq:G2 transformed} is more involved, relying on a nonlinear complex interpolation result, see \autoref{prop:berghvariante}.

\begin{lem}\label{lem:treatment of G2}
    Let $p,q \in (1,\infty)$ satisfy \eqref{eq:condition p and q}, and consider $\tw_0 \in V$ fixed.
    Then there exists $R_0 > 0$ and a constant $L > 0$ so that $\overline{B}_{\tsX_\gamma}(\tw_0,R_0) \subset V$ and
    \begin{equation*}
        | \tsG_2(\tau,\tw_1) - \tsG_2(\tau,\tw_2) | \le L(R_0,\| \tw_0 \|_{\tsX_\gamma}) \cdot \| \tw_1 - \tw_2 \|_{\tsX_\gamma}        
    \end{equation*}
    is valid for all $\tw_1, \tw_2 \in \tsX_\gamma$ with $\| \tw_i - \tw_0 \|_{\tsX_\gamma} \le R_0$, $i=1,2$, and for $\tau \in [0,T]$, where $T>0$ is sufficiently small.
\end{lem}

\begin{proof}
To simplify the notation, we define the operator
\begin{equation*}
    \sJ \colon \rW^{\eps + \nicefrac{1}{q},q}(\cD;\R^{2 \times 2}) \to \R^3, \quad g \mapsto \begin{pmatrix}
	\frac{1}{\mbody}\int_{\Gamma} g \tn(y) \,\mathrm{d} S\\
	I^{-1}\int_{\Gamma} y^\perp g \tn(y) \,\mathrm{d} S
\end{pmatrix},
\end{equation*}
for $0 < \eps < 1 - \frac{1}{q}$.
The boundedness of the trace operator $\gamma \colon \rW^{\nicefrac{\eps}{2} + \nicefrac{1}{q},q}(\cD) \to \rL^q(\partial \cD)$ yields that
\begin{equation}\label{eq:estimateofJ(g)}
	| \sJ(g) | \le C \| g \|_{\rW^{\eps + \nicefrac{1}{q},q}(\cD;\R^{2 \times 2})}, \quad g \in \rW^{\eps + \nicefrac{1}{q},q}(\cD;\R^{2 \times 2}).
\end{equation}
For $p,q \in (1,\infty)$ with $\frac{2}{p} + \frac{3}{q} < 1 - \eps$, we consider $s \in (1,2)$ with $s > \frac{2}{q} + 1$, and the resulting embedding
\begin{equation*}
    \rB_{qp}^s(\cD) \hra \rC^1(\overline{\cD}) \hra \rW^{1,q}(\cD),
\end{equation*}
see also \eqref{eq:emb Besov C1}, yields that
\begin{equation}\label{eq:Lipschitz sigma_delta 1}
    \begin{aligned}
        \| \sigma_\delta(\tv_1) - \sigma_\delta(\tv_2) \|_{\rL^q(\cD;\R^{2 \times 2})}
	    &\le C (1 + \| \tv_1 \|_{\rC^1(\cD)^4} + \| \tv_2 \|_{\rC^1(\cD)^4}) \| \tv_1 - \tv_2 \|_{\rC^1(\cD)^4}\\
	    &\le C(1 + \| \tv_1 \|_{\rB_{qp}^s(\cD)^4} + \| \tv_2 \|_{\rB_{qp}^s(\cD)^4}) \| \tv_1 - \tv_2 \|_{\rB_{qp}^s(\cD)^4}.
    \end{aligned}
\end{equation}
Analogously, we obtain that
\begin{equation}\label{eq:norm sigma_delta}
    \| \sigma_\delta(\tv_i) \|_{\rL^q(\cD;\R^{2 \times 2})} \le C(1 + \| \tv_i \|_{\rC^1(\cD)^4}) \| \tv_i \|_{\rC^1(\cD)^4} \le C (1 + \| \tv_i \|_{\rB_{qp}^s(\cD)^4}) \| \tv_i \|_{\rB_{qp}^s(\cD)^4}
\end{equation}
Moreover, for $(\xi_i,\Omega_i) \in \R^3$, we argue again that it is in particular valid that $(\xi_i,\Omega_i) \in \rW^{1,p}(0,T;\R^2) \times \rW^{1,p}(0,T)$, so we are in the framework of \autoref{rmk:procedure Z and Y from xi and Omega}.
Invoking the proof of \cite[Lemma~6.1]{GGH:13} and especially equation (6.2) in the aforementioned reference, we infer that
\begin{equation}\label{eq:estimates Q}
\begin{aligned}
    \| Q_i \|_{\rL^\infty(0,T;\R^{2 \times 2})} + \left\| Q_i^\T \right\|_{\rL^\infty(0,T;\R^{2 \times 2})} &\le C \left(1 + T(|\xi_i| + |\Omega_i|)^2\right), \quad \text{and}\\
    \| Q_1 - Q_2 \|_{\rL^\infty(0,T;\R^{2 \times 2})} &\le C T |\Omega_1 - \Omega_2|.
\end{aligned}
\end{equation}
Recalling from \autoref{sec:coordinate transform} that $\sT_\delta(\tu(t,y),\thseaice(t,y),\ta(t,y)) = Q^\T(t) \sigd(\tu(t,y),\thseaice(t,y),\ta(t,y)) Q(t)$ and making use of \eqref{eq:Lipschitz sigma_delta 1}, \eqref{eq:norm sigma_delta} as well as \eqref{eq:estimates Q}, we derive that
\begin{equation}\label{eq:est sT Lq}
    \| \sT_\delta(\tw_1) - \sT_\delta(\tw_2) \|_{\rL^q(\cD;\R^{2 \times 2})} \le C p(\| \tw_1 \|_{\rB_{qp}^s(\cD)^4 \times \R^3},\| \tw_2 \|_{\rB_{qp}^s(\cD)^4 \times \R^3}) \| \tw_1 - \tw_2 \|_{\rB_{qp}^s(\cD)^4 \times \R^3}
\end{equation}
for a suitable polynomial $p(x_1,x_2)$ with nonnegative coefficients in view of \eqref{eq:Lipschitz sigma_delta 1}, \eqref{eq:norm sigma_delta} and \eqref{eq:estimates Q}.

Similarly, it follows that
\begin{equation}\label{eq:est sT W1q}
    \| \sT_\delta(\tw_1) - \sT_\delta(\tw_2) \|_{\rW^{1,q}(\cD;\R^{2 \times 2})} \le C p(\| \tw_1 \|_{\rB_{qp}^{s+1}(\cD)^4 \times \R^3},\| \tw_2 \|_{\rB_{qp}^{s+1}(\cD)^4 \times \R^3}) \| \tw_1 - \tw_2 \|_{\rB_{qp}^{s+1}(\cD)^4 \times \R^3}
\end{equation}
for the same polynomial $p(x_1,x_2)$ as in \eqref{eq:est sT Lq}.

We define $S(\tw_1,\tw_2) := \sT_\delta(\tw_1 + \tw_2) - \sT_\delta(\tw_2)$, and we set $F_0 = \rL^q(\cD;\R^{2 \times 2}) \times \R^3$, $F_1 = \rW^{1,q}(\cD;\R^{2 \times 2})\times \R^3$, $E_0 = \rB_{qp}^s(\cD)^4 \times \R^3$ and $E_1 = \rB_{qp}^{s+1}(\cD)^4 \times \R^3$ for suitable $s \in (1,2)$ with $s > \frac{2}{q} + 1$.
Making use of \eqref{eq:est sT Lq} and \eqref{eq:est sT W1q}, we infer that for $\tw_1, \tw_2 \in E_j$ it holds that
\begin{equation*}
\begin{aligned}
    \| S(\tw_1,\tw_2) \|_{F_j} &= \| \sT_\delta(\tw_1 + \tw_2) - \sT_\delta(\tw_2) \|_{F_j}\\
    &\le C p(\| \tw_1 + \tw_2 \|_{E_j},\| \tw_2 \|_{E_j}) \| \tw_1 \|_{E_j}\\
    &\le C \Tilde{p}(\| \tw_1 \|_{E_j}, \| \tw_2 \|_{E_j}) \| \tw_1 \|_{E_j}
\end{aligned}
\end{equation*}
for another polynomial $\Tilde{p}(x_1,x_2)$ with nonnegative coefficients by virtue of the respective property of $p(x_1,x_2)$ from \eqref{eq:est sT Lq} and the triangle inequality.
Consequently, we set
\begin{equation}\label{eq:shape of omega}
    \mu(\| \tw_1 \|_{E_j}, \| \tw_2 \|_{E_j}) := C \Tilde{p}(\| \tw_1 \|_{E_j},\| \tw_2 \|_{E_j}) \| \tw_1 \|_{E_j}.
\end{equation}
The shape of $\mu$ from \eqref{eq:shape of omega} as a polynomial with nonnegative coefficients yields directly that it is continuous and positive on $\R_{\ge 0} \times \R_{\ge 0}$ and that it satisfies the monotonicity property from \autoref{prop:berghvariante}, so it lies within the scope of the latter proposition.

In order to apply this proposition, it remains to verify that $S(\tw_1,\tw_2) \in \sF(F_0,F_1)$ holds for $\tw_1, \tw_2 \in \sF(E_0,E_1)$.
If $S$ would only depend on one variable, then it would suffice to show that $S$ is Fr\'{e}chet-differentiable from $E_0$ to $F_0$ and continuous from $E_1$ to $F_1$ by virtue of 
$E_1 \hra E_0$ and $F_1 \hra F_0$, see \cite[Section~2]{Mal:89}.
The latter embeddings also result in $\Sigma_E = E_0 + E_1 = E_0$ as well as $\Sigma_F = F_0 + F_1 = F_0$, where the norm of the respective sums are equivalent to the norm of $E_0$ and $F_0$, respectively.

Let now $\tw_1, \tw_2 \in \sF(E_0,E_1)$.
Similarly as in Sections~6 and 7 of \cite{BDHH:22}, we argue that $\sigma_\delta$ is Fr\'{e}chet-differentiable from $E_0$ to $F_0$. 
Additionally recalling the shape of $\sT_\delta$ from \autoref{sec:coordinate transform} and observing that $Q$ and $Q^\T$ only depend on time and are independent of the spatial variables, we argue that $\sT_\delta$ is also Fr\'{e}chet-differentiable from $E_0$ to $F_0$.
It readily follows that $S$ is Fr\'{e}chet-differentiable from $\tilde{E}_0 := E_0 \times E_0$ to $F_0$ provided $\tilde{E}_0$ is equipped with a suitable norm.
A simple argument also reveals that $\tw = (\tw_1,\tw_2)$ is holomorphic in $\tilde{E}_0$ on $0 < \re z < 1$.
As the composition of a holomorphic function and a Fr\'{e}chet-differentiable operator is holomorphic, 
we deduce that $S(\tw_1,\tw_2)$ is holomorphic in $F_0$ on $0 < \re z < 1$.
Likewise, one can verify the other aspects of $S(\tw_1,\tw_2) \in \sF(F_0,F_1)$.

It is thus justified to employ \autoref{prop:berghvariante}, and doing so, we infer that
\begin{equation*}
    \| S(\tw_1,\tw_2) \|_{F_\theta} \le C \Tilde{p}(\| \tw_1 \|_{E_\theta},\| \tw_2 \|_{E_\theta}) \| \tw_1 \|_{E_\theta}.
\end{equation*}

Reproducing the above estimate with $S(\tw_1 - \tw_2,\tw_2)$, we get for another polynomial $p_1(x_1,x_2)$ with nonnegative coefficients
\begin{equation}\label{eq:prelimestimatesigmadelta}
	\| \sT_\delta(\tw_1) - \sT_\delta(\tw_2) \|_{F_\theta} = \| S(\tw_1 - \tw_2,\tw_2) \|_{F_\theta} \le C p_1 (\| \tw_1 \|_{E_\theta},\| \tw_2 \|_{E_\theta}) \| \tw_1 - \tw_2 \|_{E_\theta}.
\end{equation}

Next, we observe that $F_\theta = \rH^{\theta,q}(\cD;\R^{2 \times 2})$ and $E_\theta = \rB_{q p}^{s + \theta}(\cD)^4 \times \R^3$.
The assumptions on $p$ and $q$ imply that $q > 2$, so it is also valid that $\rH^{\theta,q} \hra \rW^{\theta,q}$, see e.g.\ Example~2.18 in \cite{Lun:18} for the embedding on $\R^n$ and 
use the boundary regularity to transfer this result to the present setting via an extension operator.
Plugging $\theta = \eps + \frac{1}{q}$ into \eqref{eq:prelimestimatesigmadelta}, we infer that
\begin{equation}\label{eq:finalestimateTdelta}
    \begin{aligned}
    	&\quad\| \sT_\delta(\tw_1) - \sT_\delta(\tw_2) \|_{\rW^{\eps + \nicefrac{1}{q},q}(\cD;\R^{2 \times 2})}\\
     &\le C p_1(\| \tw_1 \|_{\rB_{q p}^{s + \eps + \nicefrac{1}{q}}(\cD)^4 \times \R^3},\| \tw_2 \|_{\rB_{q p}^{s + \eps + \nicefrac{1}{q}}(\cD)^4\times \R^3}) \| \tw_1 - \tw_2 \|_{\rB_{q p}^{s + \eps + \nicefrac{1}{q}}(\cD)^4\times \R^3}.
    \end{aligned}
\end{equation}

It remains to argue that
\begin{equation*}
    \tsX_\gamma \hra \rB_{q p}^{2 - \nicefrac{2}{p}}(\cD)^4 \times \R^3 \hra \rB_{q p}^{s + \eps + \nicefrac{1}{q}}(\cD)^4 \times \R^3.
\end{equation*}
In fact, the second embedding is implied provided $s \le 2 - \frac{2}{p} - \frac{1}{q} - \eps$.
On the other hand, we recall that $s > 1 + \frac{2}{q}$ has to be ensured, so we can find such $s \in (1,2)$ if
\begin{equation*}
    1 + \frac{2}{q} < 2 - \frac{2}{p} - \frac{1}{q} - \eps, \mbox{ or, equivalently, } \frac{2}{p} + \frac{3}{q} < 1,
\end{equation*}
as $\eps$ is arbitrarily small.
The above condition is guaranteed by assumption.
The first embedding is valid in view of \autoref{lem:trace space}.

Thus, we deduce from \eqref{eq:finalestimateTdelta} that
\begin{equation}\label{eq:ultimateestimateTdelta}
	\| \sT_\delta(\tw_1) - \sT_\delta(\tw_2) \|_{\rW^{\eps + \nicefrac{1}{q},q}(\cD;\R^{2 \times 2})} \le C p_1(\| \tw_1 \|_{\tsX_\gamma},\| \tw_2 \|_{\tsX_\gamma}) \| \tw_1 - \tw_2 \|_{\tsX_\gamma}.
\end{equation}

Combining \eqref{eq:estimateofJ(g)} as well as \eqref{eq:ultimateestimateTdelta} and employing that $\| \tw_i \|_{\tsX_\gamma} \le R_0 + \| \tw_0 \|_{\tsX_\gamma}$, we conclude that
\begin{equation*}
    \begin{aligned}
        | \tsG_2(\tau,\tw_1) - \tsG_2(\tau,\tw_2) |
	&\le C \| \sT_\delta(\tw_1) - \sT_\delta(\tw_2) \|_{\rW^{\eps + \nicefrac{1}{q},q}(\cD;\R^{2 \times 2})}\\
	&\le C p_1(\| \tw_1 \|_{\tsX_\gamma},\| \tw_2 \|_{\tsX_\gamma}) \| \tw_1 - \tw_2 \|_{\tsX_\gamma}\\
	&\le C \Tilde{p}_1(R_0,\| \tw_0 \|_{\tsX_\gamma}) \cdot \| \tw_1 - \tw_2 \|_{\tsX_\gamma}
    \end{aligned}
\end{equation*}
for another polynomial $\Tilde{p}_1$ with nonnegative coefficients.
We emphasize that $\Tilde{p}_1(R_0,\| \tw_0 \|_{\tsX_\gamma})$ may depend on $T$, but it is independent of $\tau$.
\end{proof}

We then get the desired Lipschitz estimate for the complete term $\tsF$.

\begin{cor}\label{cor:Lipschitz estimate rhs F}
Let $p,q \in (1,\infty)$ be such that \eqref{eq:condition p and q} is satisfied.
Then for $\tsF$ as in \eqref{eq:rhs operator}, for $\tw_0 \in V$ and for every $R>0$ such that $\overline{B}(\tw_0,R) \subset V$, there exists $\varphi_R \in \rL^p(0,T)$ with
\begin{equation*}
    \| \tsF(\tau,\tw_1) - \tsF(\tau,\tw_2) \|_{\tsX_0} \le \varphi_R(\tau) \| \tw_1 - \tw_2 \|_{\tsX_\gamma} 
\end{equation*}
for almost all $\tau \in [0,T]$ for $T>0$ sufficiently small and all $\tw_1, \tw_2 \in \tsX_\gamma$ with $\| \tw_i - \tw_0 \|_{\tsX_\gamma} \le R$.
\end{cor}

\begin{proof}
Let $\tw_0 \in V$ and consider $R>0$ arbitrary such that $\overline{B}(\tw_0,R) \subset V$.
Moreover, we take $\tw_1, \tw_2 \in V$ with $\| \tw_i - \tw_0 \|_{\tsX_\gamma} \le R$ into account.
Simple calculations reveal that
\begin{equation}\label{eq:Omega xi perp Lipschitz}
	| \Omega_1 \xi_1^\perp - \Omega_2 \xi_2^\perp | \le C(R + | \xi_0 | + | \Omega_0 |) \left| \begin{pmatrix} \xi_1\\ \Omega_1 \end{pmatrix} - \begin{pmatrix} \xi_2\\ \Omega_2 \end{pmatrix} \right|.    
\end{equation}

Before proceeding, we remark again that we are in the setting of \autoref{rmk:procedure Z and Y from xi and Omega} and \autoref{prop:dep Z and Y on xi and Omega} by virtue of $(\xi_i,\Omega_i)$ independent of time, implying that $(\xi_i,\Omega_i) \in \rW^{1,p}(0,T;\R^2) \times \rW^{1,p}(0,T)$ for $\tw_i = (\tu_i,\thseaice_i,\ta_i,\xi_i,\Omega_i) \in V$.
With regard to $\tsM_1$ from \eqref{eq:sM}, we can then argue similarly as in the proof of \autoref{lem:lipschitz systemmatrix} to obtain
\begin{equation}\label{eq:lipschitz tsM}
    \| \tsM_1(\tau;\xi_1,\Omega_1) \tw_1 - \tsM_1(\tau,\xi_2,\Omega_2) \tw_2 \|_{\tsX_0} \le C(R + \| \tw_0 \|_{\tsX_\gamma}) (T+1) \| \tw_1 - \tw_2 \|_{\tsX_\gamma}.
\end{equation}

In total, concatenating \autoref{lem:treatment of G1}, \autoref{lem:treatment of G2}, \eqref{eq:Omega xi perp Lipschitz} and \eqref{eq:lipschitz tsM}, we derive the assertion of the corollary, where we observe that the Lipschitz constants in the aforementioned lemmas and equations are independent of $\tau$.
\end{proof}

\subsection{Proof of \autoref{thm:main theorem}}\label{ssec:proof main result}

\

Combining \autoref{lem:continuity in A1} and \autoref{lem:lipschitz systemmatrix}, we deduce that (A1) from \autoref{ass:quasilinear thm} is satisfied.
By virtue of \autoref{lem:(i)-(iii) in (A2)} and \autoref{cor:Lipschitz estimate rhs F}, aspect~(A2) from \autoref{ass:quasilinear thm} is also fulfilled.
Invoking the maximal regularity as established in \autoref{prop: max. Reg.}, showing (A3) from \autoref{ass:quasilinear thm}, we conclude by \autoref{prop:quasilinear} the existence and uniqueness of a local strong solution to \eqref{eq:ACP} on the fixed domain.
The solution lies in the corresponding maximal regularity space, namely in $\rL^p(0,T;\tsX_1) \cap \rW^{1,p}(0,T;\tsX_0)$.

From $(\xi,\Omega) \in \rW^{1,p}(0,T,\R^2 \times \rW^{1,p}(0,T)$, we obtain $\eta$, $\omega$ and $Z$ as described in \autoref{rmk:procedure Z and Y from xi and Omega}.
In total, this shows \autoref{thm:main thm transformed}.

Performing the backward change of variables and coordinates given in \autoref{sec:coordinate transform} and explained in \autoref{rmk:procedure Z and Y from xi and Omega}, we derive our main result \autoref{thm:main theorem} from \autoref{thm:main thm transformed}.
The solution must be unique as a consequence of the uniqueness implied by the quasilinear existence result and the uniqueness of the transformation. 

\medskip 

{\bf Acknowledgements}
We thank the anonymous referees for pointing out the references \cite{MT:17, HMTT:19, MSTT:19} as well as for useful comments.

\end{document}